\documentclass[12pt]{amsart}

\usepackage[latin1]{inputenc}
\usepackage{amssymb,amsmath,amsfonts, amsthm}
\usepackage{enumitem}
\usepackage{times}
\usepackage[margin=2.0cm]{geometry}
\usepackage{cite}

\newtheorem{theorem}{Theorem}[section]
\newtheorem{lemma}[theorem]{Lemma}
\newtheorem{proposition}[theorem]{Proposition}
\newtheorem{corollary}[theorem]{Corollary}

\numberwithin{equation}{section}

\newcommand{\fconvs}{{\mbox{\rm Conv}_{{\rm sc}}(\R^n)}} 
\newcommand{\fconvx}{{\mbox{\rm Conv}(\R^n)}} 
\newcommand{\fconvf}{{\mbox{\rm Conv}(\R^n; \R)}}

\newcommand{\infconv}{\mathbin{\Box}}
\newcommand{\sq}{\mathbin{\vcenter{\hbox{\rule{.3ex}{.3ex}}}}}

\newcommand{\cK}{{\mathcal K}} 
\newcommand{\cR}{\operatorname{\mathcal R}} 

\renewcommand{\S}{\mathbb{S}}
\newcommand{\sn}{{\mathbb{S}^{n-1}}}
\newcommand{\Bn}{B^n}

\newcommand{\SO}{\operatorname{SO}}

\newcommand{\Hess}{{\operatorname{D}}^2}
\newcommand{\HessTilde}{{\tilde{\operatorname{D}}}^2}
\newcommand{\D}{{\operatorname{D}}}

\newcommand{\argmin}{\operatorname{argmin}}

\newcommand{\diag}{\operatorname{diag}}

\DeclareMathOperator{\oZ}{\operatorname{Z}}

\newcommand{\oZZ}[2]{\operatorname{V}_{#1,#2}} 
\newcommand{\oZZd}[2]{\operatorname{V}_{#1,#2}^{*}} 

\newcommand{\R}{{\mathbb R}}
\newcommand{\N}{{\mathbb N}}

\newcommand{\dom}{\operatorname{dom}}
\newcommand{\hm}{\mathcal H}
\newcommand{\epi}{\operatorname{epi}}
\renewcommand{\d}{\,\mathrm{d}}
\newcommand{\ind}{{\rm\bf I}}

\newcommand{\Borel}{{\mathcal B}}
\newcommand{\Radon}{\mathcal{M}(\R^n)}

\newcommand{\Vconvf}[1]{\operatorname{VConv}_{#1}(\R^n; \R)}

\newcommand{\Had}[2]{D_{#1}^{#2}} 

\newcommand{\MA}{\text{\rm MA}} 
\newcommand{\MAp}{\text{\rm MA}^{\!*}} 
\newcommand{\nhd}[1]{\text{\rm MA}_{#1}} 
\newcommand{\nhp}[1]{\text{\rm MA}^{*}_{#1}} 

\newcommand{\hd}[1]{\Phi_{#1}} 
\newcommand{\hp}[1]{\Psi_{#1}} 

\begin{document}

\title[The Hadwiger Theorem on Convex Functions, III]{The Hadwiger Theorem on Convex Functions, III:\\ Steiner Formulas and Mixed Monge--Amp\`ere Measures}

\author{Andrea Colesanti}
\address{Dipartimento di Matematica e Informatica ``U. Dini''
Universit\`a degli Studi di Firenze,
Viale Morgagni 67/A - 50134, Firenze, Italy}
\email{andrea.colesanti@unifi.it}

\author{Monika Ludwig}
\address{Institut f\"ur Diskrete Mathematik und Geometrie,
Technische Universit\"at Wien,
Wiedner Hauptstra\ss e 8-10/1046,
1040 Wien, Austria}
\email{monika.ludwig@tuwien.ac.at}

\author{Fabian Mussnig}
\address{Dipartimento di Matematica e Informatica ``U. Dini''
Universit\`a degli Studi di Firenze,
Viale Morgagni 67/A - 50134, Firenze, Italy}
\email{mussnig@gmail.com}

\date{}

\begin{abstract} 
A complete family of functional Steiner formulas is established. 
As applications, an explicit representation of functional intrinsic volumes using special mixed Monge--Amp\`ere measures  and  a new version of the Hadwiger theorem on convex functions are obtained.

\end{abstract}

\maketitle

\section{Introduction and Statement of Results}
The classical Steiner formula states that the volume of the outer parallel set of a convex body (that is, a non-empty, compact, convex set) in $\R^n$ at distance $r>0$ can be expressed as a polynomial in $r$ of degree at most $n$.  Using that the outer parallel set of  a convex body $K\subset \R^n$ at distance $r>0$ is just the Minkowski (or vector sum) of $K$ and $r\Bn$, the ball of radius $r$, we get
\begin{equation}
\label{steiner}
V_n({K+r \Bn})=\sum_{j=0}^n r^{n-j}\kappa_{n-j} V_j(K)
\end{equation}
for every $r>0$, where $V_n$ is $n$-dimensional volume or Lebesgue measure and $\kappa_j$ is the $j$-dimensional volume of the unit ball in $\R^j$ (with the convention that $\kappa_0\!:=1$). The coefficients $V_j(K)$ are known as the \emph{intrinsic volumes} of $K$. Up to  normalization and numbering, they coincide with the classical quermassintegrals. In particular,  $V_{n-1}(K)$ is  proportional to the surface area of $K$ and $V_0(K)$ is the Euler characteristic of $K$  (that is, $V_0(K)\!:= 1$) for every convex body $K$ in $\R^n$ (cf. \cite{Schneider:CB2}). 

A complete characterization of intrinsic volumes is due to Hadwiger, who in his celebrated theorem classified all continuous, translation and rotation invariant valuations on the space, $\cK^n$,  of convex bodies in $\R^n$.
Here, we say that  $\oZ\colon\cK^n\to\R$ is a \emph{valuation} if
$$\oZ(K)+\oZ(L)=\oZ(K\cup L)+ \oZ(K\cap L)$$
for every $K,L\in\cK^n$ such that also $K\cup L\in\cK^n$. It is \emph{translation invariant} if $\oZ(\tau K)=\oZ(K)$ for every $K\in\cK^n$ and translation $\tau$ on $\R^n$ and  \emph{rotation invariant} if $\oZ(\vartheta K)=\oZ(K)$ for every $K\in\cK^n$ and $\vartheta\in \SO(n)$. The topology of $\cK^n$ is induced by the Hausdorff metric.
\begin{theorem}[Hadwiger \cite{Hadwiger:V}]\label{hugo}
A functional $\oZ\colon\cK^n\to \R$ is  a continuous, translation and rotation invariant valuation if and only if  there exist constants $\zeta_0, \dots, \zeta_n\in\R$ such that 
$$\oZ(K) = \sum_{j=0}^n \zeta_j V_j(K)$$
for every $K\in\cK^n$. 
\end{theorem}

\noindent
In addition to its many applications in convex and integral geometry  (see \cite{Hadwiger:V,Klain:Rota}), the Hadwiger theorem can be used to give a simple proof of \eqref{steiner}. 

We remark that the classification of valuations on convex bodies is a classical subject, which is described in  \cite[Chapter~6]{Schneider:CB2}. Also see \cite{HLYZ_acta, BLYZ_cpam} for some newly defined valuations and  \cite{Alesker99,Alesker01,AleskerFaifman, BernigFaifman2017,Bernig:Fu, Haberl_sln, LiMa, Ludwig:Reitzner, Ludwig:Reitzner2,Ludwig:convex, Haberl:Parapatits_centro} for some recent classification results.

\goodbreak
Recently, the authors \cite{Colesanti-Ludwig-Mussnig-5} introduced functional intrinsic volumes on convex functions. Let
$$\fconvs:=\Big\{u\colon\R^n\to(-\infty,+\infty]\colon u \not\equiv +\infty, \lim_{|x|\to+\infty} \frac{u(x)}{|x|}=+\infty, u \text{ is l.s.c. and convex}\Big\}$$
denote the space of all proper, super-coercive, lower semicontinuous, convex functions on $\R^n$, where $\vert \cdot\vert$ denotes the Euclidean norm. For $\zeta\in C_b((0,\infty))$, the set of continuous functions with bounded support on $(0,\infty)$, and $0\leq j \leq n$, consider the functional
\begin{equation}
\label{eq:funct_int_vol_smooth}
u\mapsto \int_{\R^n} \zeta(|\nabla u(x)|) \big[\Hess u(x)\big]_{n-j} \d x
\end{equation}
on $\fconvs\cap C_+^2(\R^n)$, where $C_+^2(\R^n)$ denotes the set of $u\in C^2(\R^n)$ with positive definite Hessian matrix $\Hess u$  and $[A]_k$ is the $k$th elementary symmetric function of the eigenvalues of the symmetric $n\times n$ matrix $A$ (with the convention that $[A]_0\!:=1$).

Under suitable conditions on the function $\zeta$, the functional \eqref{eq:funct_int_vol_smooth} 
continuously extends to the whole space $\fconvs$. Here, continuity is understood with respect to epi-convergence (see Subsection~\ref{ss:convex functions}). In case $\zeta$ can be identified with an element of $C_c([0,\infty))$, the set of continuous functions with compact support on $[0,\infty)$, it was shown in \cite{Colesanti-Ludwig-Mussnig-4} that \eqref{eq:funct_int_vol_smooth} continuously extends to $\fconvs$ by using Hessian measures (see Subsection \ref{ss:hessian_measures_fconvs} for the definition).

More recently, the authors proved that  \eqref{eq:funct_int_vol_smooth} continuously extends for the following classes of singular densities $\zeta$. For $0\leq j \leq n-1$, let
$$\Had{j}{n}:=\Big\{\zeta\in C_b((0,\infty))\colon \lim_{s\to 0^+} s^{n-j} \zeta(s)=0, \lim_{s\to 0^+} \int_s^{\infty} t^{n-j-1} \zeta(t) \d t \text{ exists and is finite} \Big\}.$$
In addition, let $\Had{n}{n}$ be the set of all functions $\zeta\in C_b((0,\infty))$ such that $\lim_{s\to 0^+} \zeta(s)$ exists and is finite. For $\zeta\in\Had{n}{n}$, we set $\zeta(0)\!:=\lim_{s\to 0^+} \zeta(s)$ and identify $\zeta$ with the corresponding element of $C_c([0,\infty))$. 

\begin{theorem}[\!\cite{Colesanti-Ludwig-Mussnig-5}, Theorem 1.2]
\label{thm:existence_singular_hessian_vals}
For $0\leq j \leq n$ and $\zeta\in\Had{j}{n}$, there exists a unique, continuous, epi-translation and rotation invariant valuation $\oZZ{j}{\zeta}\colon\fconvs\to\R$ such that
\begin{equation}
\label{eq:rep_ozz_c2}
\oZZ{j}{\zeta}(u)=\int_{\R^n} \zeta(|\nabla u(x)|)\big[\Hess u(x)\big]_{n-j} \d x
\end{equation}
for every $u\in\fconvs\cap C_+^2(\R^n)$.
\end{theorem}

\noindent
Here, we say that $\oZ \colon \fconvs\to\R$ is a \emph{valuation} if
$$\oZ(u)+\oZ(v)=\oZ(u\vee v)+\oZ(u\wedge v)$$
for every $u,v\in \fconvs$ such that also their pointwise maximum $u\vee v$ and minimum $u\wedge v$ belong to $\fconvs$. A valuation $\oZ\colon\fconvs\to\R$ is said to be \emph{epi-translation invariant} if $\oZ(u\circ \tau^{-1}+\gamma)=\oZ(u)$ for every translation $\tau$ on $\R^n$, every $\gamma\in\R$ and every $u\in \fconvs$ and it is \emph{rotation invariant} if $\oZ(u\circ \vartheta^{-1})=\oZ(u)$ for every $\vartheta\in\SO(n)$ and $u\in\fconvs$. We remark that these properties are natural extensions of the corresponding properties of the classical intrinsic volumes.

A closed representation of the extensions of \eqref{eq:rep_ozz_c2} to $\fconvs$ was obtained for the cases $j=0$ and $j=n$. For $\zeta\in\Had{0}{n}$, the functional $\oZZ{0}{\zeta}$ is a constant, independent of $u\in\fconvs$, and for $\zeta\in\Had{n}{n}$, we have
\begin{equation}
\label{eq:n-hom}
\oZZ{n}{\zeta}(u)=\int_{\dom u} \zeta(|\nabla u(x)|)\d x
\end{equation}
for every $u\in\fconvs$, where $\dom u\!:=\{x\in\R^n\colon u(x)<\infty\}$ is the \emph{domain} of $u$ (see \cite[Theorem 2]{Colesanti-Ludwig-Mussnig-4}). However, apart from these extremal cases, the functionals $\oZZ{j}{\zeta}$ were so far only described as continuous extensions of \eqref{eq:rep_ozz_c2} and by  Cauchy--Kubota formulas, which were recently established in \cite[Theorem 1.6]{Colesanti-Ludwig-Mussnig-6}.

\goodbreak
In \cite{Colesanti-Ludwig-Mussnig-5}, the following functional Hadwiger theorem  was established. 
\begin{theorem}[\!\!\cite{Colesanti-Ludwig-Mussnig-5}, Theorem 1.3]
\label{thm:hadwiger_convex_functions}
A functional $\oZ\colon\fconvs \to \R$ is a continuous, epi-translation and rotation invariant valuation if and only if there exist functions $\zeta_0\in\Had{0}{n}$, \dots, $\zeta_n\in\Had{n}{n}$  such that
\begin{equation*}
\oZ(u)= \sum_{j=0}^n \oZZ{j}{\zeta_j}(u) 
\end{equation*}
for every $u\in\fconvs$.
\end{theorem}
\noindent
Using the notion of epi-homogeneity of degree $j$ (see Subsection~\ref{ss:convex functions}), we see that Theorems \ref{hugo} and \ref{thm:hadwiger_convex_functions} imply that for $0\le j\le n$, the functionals $\oZZ{j}{\zeta}$ for $\zeta\in\Had{j}{n}$ correspond to multiples of the classical intrinsic volumes $V_j$. Hence, we call  $\oZZ{j}{\zeta}$ for $0\le j\le n$ and $\zeta\in\Had{j}{n}$ a $j$th \emph{functional intrinsic volume}. Moreover,  the family
$\{\oZZ{j}{\zeta}\colon \zeta\in\Had{j}{n}\}$
describes all continuous, epi-translation and rotation invariant valuations on $\fconvs$ that are epi-homogeneous of degree $j$ and is, in this sense, canonical.

We remark that the classification of valuations on function spaces has only been started to be studied recently. The first classification results for valuations on classical function spaces were obtained for $L_p$ and Sobolev spaces, and  for Lipschitz and continuous functions (see  \cite{Tsang:Lp, Ludwig:SobVal,Ludwig:Fisher,ColesantiPagniniTradaceteVillanueva,ColesantiPagniniTradaceteVillanueva2021, Villanueva2016}). Results on valuations on convex functions can be found in \cite{ColesantiLudwigMussnig17,ColesantiLudwigMussnig,Mussnig19, Mussnig21,Alesker_cf,Colesanti-Ludwig-Mussnig-4, Knoerr1, Knoerr2}. 

In this article we present a new, complete family of Steiner formulas for functional intrinsic volumes and its applications. 
For $\zeta\in\Had{n}{n}$ (or equivalently, $\zeta\in C_c([0,\infty))$), the \emph{functional Steiner formula} is the following result.
\begin{theorem}
\label{thm_steiner_functions}
If $\zeta\in\Had{n}{n}$, then
\begin{equation}
\label{steiner_function}
\oZZ{n}{\zeta}(u\infconv(r\sq \ind_{\Bn})) = \sum_{j=0}^n r^{n-j} \kappa_{n-j} \oZZ{j}{\zeta_j}(u)
\end{equation}
for every $u\in\fconvs$ and $r>0$, where $\zeta_j\in \Had{j}{n}$ is given by
\begin{equation}
\label{eq:steiner_function_transform}
\zeta_j(s):=\frac{1}{\kappa_{n-j}} \left(\frac{\zeta(s)}{s^{n-j}} - (n-j) \int_s^{\infty} \frac{\zeta(t)}{t^{n-j+1}} \d t \right)
\end{equation}
for $s>0$ and $0\leq j \leq n$.
\end{theorem}
\noindent
Here, $u\infconv w$ is the infimal convolution of $u,w\in\fconvs$ and $r\sq w$ is obtained by epi-multiplication of $w$ with $r>0$ while  $\ind_{\Bn}$ is the convex indicator function of the Euclidean unit ball $\Bn$ (see Subsection~\ref{ss:convex functions} for the precise definitions). Note that
\begin{equation}\label{epi-add}
\epi (u\infconv(r\sq \ind_{\Bn}))= \epi u +r (\Bn\times \R),
\end{equation}
where $\epi w\!:=\{(x,t)\in\R^n\times\R: t\ge w(x)\}$ is the epi-graph of $w\colon \R^n\to (-\infty,+\infty]$ and the addition on the right side of \eqref{epi-add} is Minkowski addition in $\R^n\times \R$.

\goodbreak
We give two proofs of Theorem \ref{thm_steiner_functions}. In Section \ref{se:proofs}, we give a direct proof (not using the functional Hadwiger theorem, Theorem \ref{thm:hadwiger_convex_functions}) and in Section \ref{additional results}, we prove Theorem \ref{thm_steiner_functions} using Theorem \ref{thm:hadwiger_convex_functions}. This corresponds to the fact that the classical Steiner formula can be proved both directly and as a consequence of the Hadwiger theorem.

Equation \eqref{steiner_function}  corresponds to the classical Steiner formula \eqref{steiner}.
Indeed, we will see that \eqref{steiner} can be easily retrieved from \eqref{steiner_function}. Furthermore, by properties of the transform \eqref{eq:steiner_function_transform}, every functional intrinsic volume $\oZZ{j}{\zeta_j}$ for $1\leq j\leq n$ and $\zeta_j\in\Had{j}{n}$ will appear exactly once on the right side of \eqref{steiner_function} as $\zeta$ ranges in $\Had{n}{n}$. 
In this sense, Theorem~\ref{thm_steiner_functions} provides a complete description of functional intrinsic volumes on $\fconvs$. 
We remark that Steiner formulas for convex functions are also obtained if we replace $\ind_{\Bn}$ in \eqref{steiner_function} by other radially symmetric, super-coercive, convex functions. However, in general such formulas do not give rise to all functional intrinsic volumes. For more details, see Subsection~\ref{ss:further_steiner_formulas}.

\goodbreak
As an immediate consequence of Theorem~\ref{thm_steiner_functions}, equation \eqref{eq:n-hom} and properties of the transform \eqref{eq:steiner_function_transform} (see Lemma~\ref{le:r_kln}), we obtain the following new representation of the functionals $\oZZ{j}{\zeta}$.

\begin{corollary}
\label{co:def_steiner}
If $\,0\leq j < n$ and $\zeta\in\Had{j}{n}$, then
\begin{align*}
\oZZ{j}{\zeta}(u) &= \frac{j!}{n!} \frac{\d^{n-j}}{\d r^{n-j}} \Big\vert_{r=0} \oZZ{n}{\alpha}(u\infconv(r \sq \ind_{\Bn}))\\
&= \frac{j!}{n!} \frac{\d^{n-j}}{\d r^{n-j}}\Big\vert_{r=0} \int_{\dom (u \infconv (r \sq \ind_{\Bn}))} \alpha\left(\big\vert \nabla\big(u \infconv (r \sq \ind_{\Bn}) \big)\big\vert\right) \d x
\end{align*}
for every $u\in\fconvs$, where $\alpha\in C_c([0,\infty))$ is given by
$$\alpha(s):= \binom{n}{j} \Big( s^{n-j} \zeta(s) + (n-j) \int_s^{\infty} t^{n-j-1} \zeta(t) \d t \Big)$$
for $s>0$.
\end{corollary}
\goodbreak

Using a new family of measures, we establish new closed representations of the functional intrinsic volumes on the whole space $\fconvs$ that do not require singular densities. For $u\in\fconvs$, let $\MAp(u;\cdot)$ be the push-forward through  $\nabla u$  of $n$-dimensional Lebesgue measure restricted to the domain of $u$. Equivalently, $\MAp(u;\cdot)$ is the Monge--Amp\`ere measure of the convex conjugate of $u$ (see Section~\ref{se:conjugate_mma_measures} for details) and we call it the \emph{conjugate Monge--Amp\`ere measure} of $u$. For functions $u_1,\dots,u_n\in\fconvs$, we write $\MAp(u_1,\dots,u_n;\cdot)$ for the polarization of $\MAp(u;\cdot)$ with respect to infimal convolution (see Section~\ref{se:conjugate_mma_measures}) and call $\MAp(u_1,\dots,u_n;\cdot)$ the \emph{conjugate mixed Monge--Amp\`ere measure} of $u_1,\dots,u_n$. For $0\leq j \leq n$ and $u\in\fconvs$, we set
\begin{equation}
\label{eq:def_nhp_j}
\nhp{j}(u;\cdot):=\MAp(u[j],\ind_{\Bn}[n-j];\cdot),
\end{equation}
where the function $u$ is repeated $j$ times and the convex indicator function $\ind_{\Bn}$ is repeated $(n-j)$ times. We establish the following result.

\begin{theorem}
\label{thm:steiner_measures}
If $\,0\leq j \leq n$ and $\zeta\in\Had{j}{n}$, then
$$\oZZ{j}{\zeta}(u)=\int_{\R^n} \alpha(|y|) \d\nhp{j}(u;y)$$
for every $u\in\fconvs$, where $\alpha\in C_c([0,\infty))$ is given by
$$\alpha(s):=\binom{n}{j} \Big(s^{n-j}\zeta(s) + (n-j) \int_s^{\infty}t^{n-j-1} \zeta(t) \d t \Big)$$
for $s>0$. Moreover, for $1\leq j \leq n$,
\begin{equation}
\label{eq:steiner_measures_smooth}
\oZZ{j}{\zeta}(u)=\frac{1}{\binom{n}{j}}\int_{\R^n} \alpha(|\nabla u(x)|)\,\tau_{n-j}(u,x)\d x
\end{equation}
for $u\in\fconvs\cap C^2_+(\R^n)$.
\end{theorem}
\noindent
Here, for $u\in\fconvs\cap C_+^2(\R^n)$ and $0\leq i\leq n-1$, we write $\tau_i(u,x)$ for the $i$th elementary symmetric function of the principal curvatures of the sublevel set
$\{y\in\R^n\colon u(y)\leq t\}$
at $x\in\R^n$  with $t=u(x)$ (and we use the convention $\tau_0(u,x)\!:=1$). Note that $\tau_i(u,x)$ is well-defined for such $u$ if $u(x)> \min_{y\in\R^n} u(y)$. Since such $u$ attains  its minimum at only one point, the integral in \eqref{eq:steiner_measures_smooth} is also well-defined.
We remark that a direct  proof of \eqref{eq:steiner_measures_smooth} was given in \cite[Lemma 3.9]{Colesanti-Ludwig-Mussnig-6}. Here it is a consequence of properties of the measures $\nhp{j}(u;\cdot)$ (see Theorem~\ref{representation of Steiner measures primal smooth}).

\goodbreak
Conjugate mixed Monge--Amp\`ere measures generalize Hessian measures on  $\fconvs$ (see Subsection~\ref{ss:hessian_measures_fconvs}) 
and the precise connection of integrals involving the measure $\nhp{j}(u;\cdot)$ and Hessian measures for $u\in\fconvs$ is established in Section \ref{se:steiner}. It is the basis of a new proof of Theorem~\ref{thm:existence_singular_hessian_vals} presented in Section~\ref{se:proofs}, where we also prove Theorem~\ref{thm:steiner_measures}.
\goodbreak

Combining Theorem \ref{thm:hadwiger_convex_functions} and Theorem~\ref{thm:steiner_measures}, we obtain the following new version of the Hadwiger theorem for convex functions. 

\begin{theorem}
A functional $\oZ\colon\fconvs\to\R$ is a continuous, epi-translation and rotation invariant valuation if and only if there exist functions $\alpha_0,\dots,\alpha_n\in C_c([0,\infty))$ such that
$$\oZ(u)=\sum_{j=0}^n \int_{\R^n} \alpha_j(|y|) \d \nhp{j}(u;y)$$
for every $u\in\fconvs$.
\end{theorem}

\noindent
By properties of the integral transform from Theorem~\ref{thm:steiner_measures}  which maps $\zeta$ to $\alpha$, this version is equivalent to Theorem~\ref{thm:hadwiger_convex_functions}.

Using the Legendre--Fenchel transform or convex conjugate, we can translate the new results on $\fconvs$ to results on 
$\fconvf \!:=\{v\colon\R^n\to \R\colon v \text{ is convex}\}$, 
the space of finite-valued convex functions on $\R^n$. In fact, most results will be proved on $\fconvf$ and then transferred to $\fconvs$ using convex conjugation. Results on $\fconvf$ are presented in Section~\ref{se:results_on_fconvf}.
The next section is devoted to notation and preliminaries. In Section \ref{se:ma}, results on Monge--Amp\`ere measures and mixed Monge--Amp\`ere measures on $\fconvf$ are collected and the new measures $\nhd{j}(v;\cdot)$ for $v\in\fconvf$ and $0\le j\le n$ are discussed. In Section \ref{se:conjugate_mma_measures}, the corresponding results are presented on $\fconvs$. Results connecting the measure $\nhd{j}(v;\cdot)$ to the $j$th Hessian measure of $v\in\fconvf$ are established in Section \ref{se:steiner}. In the following section, the proofs of the main results are presented. In the final section, an alternate proof of the functional Steiner formula, results on the explicit representation of functional intrinsic volumes and on the retrieval of classical intrinsic volumes are presented. Moreover, general functional Steiner formulas are discussed.

\section{Results for Valuations on Finite-valued Convex Functions}
\label{se:results_on_fconvf}
A functional $\oZ\colon\fconvf\to\R$ is \emph{dually epi-translation invariant} if and only if $\oZ(v+\ell+\gamma)=\oZ(v)$ for every $v\in\fconvf$, every linear functional $\ell\colon\R^n\to\R$ and every $\gamma\in\R$, or equivalently, if the map $u\mapsto \oZ(u^*)$, defined on $\fconvs$, is epi-translation invariant. 
It was shown in \cite{Colesanti-Ludwig-Mussnig-3} that $\oZ\colon\fconvf\to\R$ is a continuous valuation if and only if $u\mapsto \oZ(u^*)$ is a continuous valuation on $\fconvs$ (see Proposition~\ref{basics}).

The following result is equivalent to Theorem \ref{thm:existence_singular_hessian_vals} by duality.

\begin{theorem}[\!\cite{Colesanti-Ludwig-Mussnig-5}, Theorem 1.4]
\label{thm:existence_singular_hessian_vals_dual}
For $0\leq j \leq n$ and $\zeta\in\Had{j}{n}$, there exists a unique, continuous, dually epi-translation and rotation invariant valuation $\oZZd{j}{\zeta}\colon\fconvf\to\R$ such that
\begin{equation}
\label{eq:rep_ozz_c2_dual}
\oZZd{j}{\zeta}(v)=\int_{\R^n} \zeta(|x|)\big[\Hess v(x)\big]_{j} \d x
\end{equation}
for every $v\in\fconvf\cap C^2_+(\R^n)$.
\end{theorem}

\noindent
Here, for $0\leq j\leq n$ and $\zeta\in\Had{j}{n}$, the valuation $\oZZd{j}{\zeta}$ is dual to $\oZZ{j}{\zeta}$ in the sense that $\oZZd{j}{\zeta}(v)=\oZZ{j}{\zeta}(v^*)$ for every $v\in\fconvf$. We remark that the new proof of Theorem~\ref{thm:existence_singular_hessian_vals_dual} that we present in Section~\ref{se:proofs} actually shows that the representation \eqref{eq:rep_ozz_c2_dual} holds on $\fconvf\cap C^2(\R^n)$.

\goodbreak
The Hadwiger Theorem on $\fconvf$ is the following result, which is equivalent to Theorem~\ref{thm:hadwiger_convex_functions} by duality.

\begin{theorem}[\!\!\cite{Colesanti-Ludwig-Mussnig-5}, Theorem 1.5]
\label{thm:hadwiger_convex_functions_dual}
A functional $\oZ\colon\fconvf \to \R$ is a continuous, dually epi-translation and rotation invariant valuation if and only if there exist functions $\zeta_0\in\Had{0}{n}$, \dots, $\zeta_n\in\Had{n}{n}$  such that
\begin{equation*}
\oZ(v)= \sum_{j=0}^n \oZZ{j}{\zeta_j}^*(v) 
\end{equation*}
for every $v\in\fconvf$.
\end{theorem}

\goodbreak
We obtain the following dual version of the functional Steiner formulas from Theorem~\ref{thm_steiner_functions}. We use the support function of the unit ball, $h_{\Bn}(x)=|x|$ for $x\in\R^n$, and the fact that $(u\infconv (r\sq \ind_{\Bn}))^*=u^* + r\, h_{\Bn}$ for $u\in\fconvs$ and $r>0$.

\goodbreak
\begin{theorem}
\label{thm_steiner_functions_dual}
If $\zeta\in\Had{n}{n}$, then
\begin{equation*}
\oZZd{n}{\zeta}(v+ r\, h_{\Bn}) = \sum_{j=0}^n r^{n-j} \kappa_{n-j} \oZZd{j}{\zeta_j}(v)
\end{equation*}
for every $v\in\fconvf$ and $r>0$, where $\zeta_j\in \Had{j}{n}$ is given by
\begin{equation}
\label{eq:steiner_function_dual_transform}
\zeta_j(s):=\frac{1}{\kappa_{n-j}} \left(\frac{\zeta(s)}{s^{n-j}} - (n-j) \int_s^{\infty} \frac{\zeta(t)}{t^{n-j+1}} \d t \right)
\end{equation}
for $s>0$ and $\,0\leq j \leq n$.
\end{theorem}

An immediate consequence is the following result.

\begin{corollary}
Let $0\leq j < n$. If $\zeta\in\Had{j}{n}$, then
$$\oZZd{j}{\zeta}(v) = \frac{j!}{n!} \frac{\d^{n-j}}{\d r^{n-j}} \Big\vert_{r=0} \oZZd{n}{\alpha}(v + r\, h_{\Bn})$$
for every $v\in\fconvf$, where $\alpha\in C_c([0,\infty))$ is given by
$$\alpha(s):= \binom{n}{j} \Big(s^{n-j} \zeta(s) + (n-j) \int_s^{\infty} t^{n-j-1} \zeta(t) \d t\Big)$$
for $s>0$.
\end{corollary}

Let $\MA(v;\cdot)$ be the Monge--Amp\`ere measure of $v\in\fconvf$ and write $\MA(v_1,\dots,v_n;\cdot)$ for its polarization, the \emph{mixed Monge--Amp\`ere measure} of $v_1,\dots,v_n\in\fconvf$. 
For $0\leq j \leq n$ and $v\in\fconvf$, we set
$$\nhd{j}(v;\cdot):=\MA(v[j], h_{\Bn}[n-j];\cdot)$$
(see Section \ref{se:ma} for results on Monge--Amp\`ere measures, mixed Monge--Amp\`ere measures and this new family of measures).
\goodbreak
The following result corresponds to Theorem~\ref{thm:steiner_measures}.

\begin{theorem}
\label{thm:steiner_measures_dual}
If $\,0\leq j \leq n$ and $\zeta\in\Had{j}{n}$, then
$$\oZZd{j}{\zeta}(v)=\int_{\R^n} \alpha(|x|) \d\nhd{j}(v;x)$$
for every $v\in\fconvf$, where $\alpha\in C_c([0,\infty))$ is given by
$$\alpha(s):=\binom{n}{j} \Big(s^{n-j}\zeta(s) + (n-j) \int_s^{\infty}t^{n-j-1} \zeta(t) \d t \Big)$$
for $s>0$. Moreover, for $1\leq j \leq n$,
\begin{equation}
\label{eq:rep_ozz_j_zeta_dual}
\oZZd{j}{\zeta}(v)=\int_{\R^n} \alpha(|x|)\det(\Hess v(x)[j],\Hess h_{\Bn}(x)[n-j])\d x
\end{equation}
for $v\in\fconvf\cap C^2(\R^n)$.
\end{theorem}

\noindent
Here, $\det(A_1,\dots,A_n)$ denotes the mixed discriminant of the symmetric $n\times n$ matrices $A_1,\dots, A_n$. Note that $\Hess h_{\Bn}(x)$ exists for every $x\ne0$ and that \eqref{eq:rep_ozz_j_zeta_dual} is well-defined as a Lebesgue integral. 
Combining Theorem \ref{thm:hadwiger_convex_functions_dual} and Theorem~\ref{thm:steiner_measures_dual}, we obtain the following new version of the Hadwiger theorem for finite-valued convex functions.

\begin{theorem}
A functional $\,\oZ\colon\fconvf\to\R$ is a continuous, dually epi-translation and rotation invariant valuation if and only if there exist functions $\alpha_0,\ldots,\alpha_n\in C_c([0,\infty))$ such that
$$\oZ(v)=\sum_{j=0}^n \int_{\R^n} \alpha_j(|x|) \d \nhd{j}(v;x)$$ 
for every $v\in\fconvf$.
\end{theorem}

\noindent
By properties of the integral transform from Theorem~\ref{thm:steiner_measures_dual}  which maps $\zeta$ to $\alpha$, this version is equivalent to Theorem~\ref{thm:hadwiger_convex_functions_dual}.

\section{Preliminaries}
We work in $n$-dimensional Euclidean space $\R^n$, with $n\ge 1$, endowed with the Euclidean norm $\vert \cdot\vert $ and the usual scalar product
$\langle \cdot,\cdot\rangle$. We also use coordinates, $x=(x_1,\dots,x_n)$, for $x\in\R^n$. Let $\Bn\!:=\{x\in\R^n\colon \vert x\vert \le 1\}$ be the Euclidean unit ball and $\sn$ the unit sphere in $\R^n$. 
A basic reference on convex bodies is the book by Schneider \cite{Schneider:CB2}.

\subsection{Mixed Discriminants}
We will need some basic definitions and properties which can be found in Section 5.5 of the book by Schneider \cite{Schneider:CB2}.
Given symmetric $n\times n$ matrices $A_k = (a_{ij}^k)$ for $1\leq k \leq n$, their \emph{mixed discriminant} is defined as
$$\det(A_1,\ldots,A_n) := \frac{1}{n!} \sum_{\sigma} \,\det \left(
\begin{array}{ccc}
a_{11}^{\sigma(1)} & \cdots & a_{1n}^{\sigma(n)} \\
\vdots &  & \vdots \\
a_{n1}^{\sigma(1)} & \cdots & a_{nn}^{\sigma(n)} \\
\end{array}
\right)
$$
where we sum over  all permutations $\sigma$ of $\{1,\ldots,n\}$. As a consequence of this definition, the mixed discriminant $\det$ is multilinear and symmetric in its entries. Alternatively, the mixed discriminant is uniquely determined as the symmetric functional that satisfies
\begin{equation}\label{polarisation md}
\det(\lambda_1 A_1 + \cdots + \lambda_m A_m) = \sum_{i_1,\ldots,i_n=1}^m \lambda_{i_1}\cdots \lambda_{i_n} \det(A_{i_1},\ldots,A_{i_n})
\end{equation}
for all $\lambda_1,\ldots,\lambda_m\in\R$, symmetric $n\times n$ matrices $A_1,\ldots,A_m$ and $m\ge1$. By the polarization formula, the mixed determinant can also be written as
\begin{equation}\label{polarization_md}
\det(A_1,\dots,A_n)=\frac1{n!}\sum_{k=1}^n\ \sum_{1\le j_1<\dots<j_k\le n}(-1)^{n-k}\det(A_{j_1}+\dots+A_{j_k})
\end{equation}
for symmetric $n\times n$ matrices $A_1,\ldots,A_n$ (see, for example, \cite[Theorem 4]{Bapat}).
In addition,
there exist maps $D_{ij}\colon(\R^{n\times n})^{n-1}\to\R$ for $1\leq i,j\leq n$ such that
\begin{equation}\label{Dij}
\det(A_1,\ldots,A_n)=\sum_{i,j=1}^n D_{ij}(A_1,\ldots,A_{n-1})\,a_{ij}^n
\end{equation}
for all symmetric $n\times n$ matrices $A_1,\ldots,A_n$.
We remark that it follows from \eqref{polarisation md} that
\begin{equation}
\label{eq:mixed_dis_hessian}
[A]_j =\binom{n}{j}\det(A[j],I_n[n-j])
\end{equation}
for every symmetric $n\times n$ matrix $A$, where $I_n$ is the $n\times n$ identity matrix. If the symmetric matrix $A$ is, in addition, invertible, then 
\begin{equation}
\label{eq:jtrace}
[A]_j=\det(A)\,[A^{-1}]_{n-j}
\end{equation}
for $0\le j\le n$.

\subsection{Convex Functions}\label{ss:convex functions}
We collect some basic results and properties of convex functions. Standard references are the books by Rockafellar \cite{Rockafellar} and Rockafellar \& Wets \cite{RockafellarWets}.

Let $\fconvx$  be the set of proper, lower semicontinuous, convex functions $u\colon\R^n\to(-\infty,\infty]$, where $u$ is proper if $u\not\equiv +\infty$. 
For $t\in\R$, we write
$$\{u\leq t\}:=\{x\in\R^n\colon u(x)\leq t\}$$
for the \emph{sublevel sets} of $u$. If $u\in\fconvs$, then $u$ attains its minimum and 
we set 
$$\argmin u:=\{x\in\R^n\colon u(x)=\min\nolimits_{z\in\R^n} u(z)\}.$$
This is a convex body which, if in addition $u\in C^2_+(\R^n)$, consists of a single point.

The standard topology on $\fconvx$ and its subsets is induced by epi-convergence. A sequence of functions $u_k\in\fconvx$ is \emph{epi-convergent} to $u\in\fconvx$ if for every $x\in\R^n$:
\begin{enumerate}
    \item[(i)] $u(x)\leq \liminf_{k\to\infty} u_k(x_k)$ for every sequence $x_k\in\R^n$ that converges to $x$;
    \item[(ii)] $u(x)=\lim_{k\to\infty} u_k(x_k)$ for at least one sequence $x_k\in\R^n$ that converges to $x$.
\end{enumerate}
Note that the limit of an epi-convergent sequence of functions from $\fconvx$ is always lower semi\-continuous. 

\goodbreak
For $u\in\fconvx$, we denote by $u^*\in\fconvx$ its \emph{Legendre--Fenchel transform} or \emph{convex conjugate}, which is defined by
$$u^*(y):=\sup\nolimits_{x\in\R^n} \big(\langle x,y \rangle - u(x) \big)$$
for $y\in\R^n$. Since $u$ is lower semicontinuous, we have $u^{**}=u$. 
For a convex body $K\in\cK^n$, we denote by $\ind_K\in\fconvs$ its \emph{convex indicator function}, which is defined as
$$\ind_K(x):=\begin{cases}
0\quad &\text{for } x\in K,\\
+\infty\quad &\text{for } x\not\in K.\end{cases}$$
\goodbreak\noindent
We have
$$\ind_K^*= h_K,$$
where $h_K\colon\R^n\to \R$ is the \emph{support function} of $K$, defined as
$$h_K(x):=\max\nolimits_{y\in K}\langle x,y \rangle.$$

For  $u_1,u_2\in\fconvs$, we denote by $u_1\infconv u_2\in\fconvs$ their \emph{infimal convolution} or \emph{epi-sum} which is defined as
$$(u_1\infconv u_2)(x):=\inf\nolimits_{x_1+x_2=x} u_1(x_1)+u_2(x_2)$$
for $x\in\R^n$. 
The \emph{epi-multiplication} of $u\in\fconvs$ by $\lambda>0$ is defined in the following way.  We set
$$\lambda\sq u(x):=\lambda\, u\left( \frac x\lambda \right)$$
for $x\in\R^n$ and note that $\lambda\sq u\in\fconvs$. This corresponds to rescaling the epi-graph of $u$ by the factor $\lambda$, that is, $\epi \lambda\sq u=\lambda \epi u$.

\begin{proposition}\label{basics} The following properties hold.
\begin{enumerate}[label=\emph{(\alph*)}, ref={(\alph*)}, leftmargin=18pt]
\item The function $u\in\fconvs$ if and only if $u^*\in\fconvf$.
\item The function $u\in\fconvs\cap C^2_+(\R^n)$ if and only if $u^*\in \fconvs\cap C^2_+(\R^n)$.  
\item If $u_1, u_2\in\fconvs$ are such that $u_1\vee u_2$ and $u_1\wedge u_2$ are in $\fconvs$, then $u_1^*\vee u_2^*$ and $ u_1^*\wedge u_2^*$ are in $\fconvf$ and
$$
(u_1\vee u_2)^*=u_1^*\wedge u_2^*,\quad\quad (u_1\wedge u_2)^*=u_1^*\vee u_2^*.
$$
\item For $u_1, u_2\in\fconvs$ and  $\lambda_1,\lambda_2>0$,
$$
(\lambda_1\sq u_1\infconv \lambda_2\sq u_2)^*=\lambda_1 u_1^*+\lambda_2 u_2^*.
$$
\item The sequence $u_k$ in $\fconvs$ epi-converges to $u\in\fconvs$, if and only if the sequence $u^*_k$ in $\fconvf$ epi-converges to $u^*\in\fconvf$.

\end{enumerate}
\end{proposition}

We say that a functional $\oZ\colon \fconvs\to\R$ is \emph{epi-homogeneous of degree $j$} if
$$\oZ(\lambda \sq u)=\lambda^j \oZ(u)$$
for every $\lambda >0$ and $u\in\fconvs$. A functional $\oZ\colon \fconvf\to\R$ is \emph{homogeneous of degree $j$} if
$$\oZ(\lambda v)=\lambda^j \oZ(v)$$
for every $\lambda >0$ and $v\in\fconvf$. It is a consequence of Proposition~\ref{basics} that a map\linebreak$\oZ:\fconvs\to\R$ is a continuous valuation that is epi-homogeneous of degree $j$ if and only if $v\mapsto \oZ(v^*)$ is a continuous valuation on $\fconvf$ that is homogeneous of degree $j$. We say that $\oZ\colon\fconvs\to \R$ is \emph{epi-additive} if
$$\oZ(u_1\infconv u_2)=\oZ(u_1)+\oZ(u_2)$$
for every $u_1, u_2\in \fconvs$. The dual notion is additivity on $\fconvf$, where a functional $\oZ\colon\fconvf\to \R$ is \emph{additive} if
$$\oZ(v_1+ v_2)=\oZ(v_1)+\oZ(v_2)$$
for every $v_1, v_2\in \fconvf$.

\subsection{The Integral Transform $\cR$}\label{ss:integral transform}

In \cite{Colesanti-Ludwig-Mussnig-6}, the integral transform $\cR$ and its inverse $\cR^{-1}$ were introduced. For $\zeta\in C_b((0,\infty))$ and $s>0$, let
\begin{equation*}
\cR \zeta(s):= s\, \zeta(s) + \int_s^{\infty}  \zeta(t)\d t.
\end{equation*}
It is easy to see that also $\cR\zeta\in C_b((0,\infty))$.

For $l\ge 1$, we write
$$\cR^l \zeta:=\underbrace{(\cR \circ \cdots \circ \cR)}_{l} \zeta$$ 
and set $\cR^0 \zeta \!:= \zeta$.  We set $\cR^{-l}=(\cR^{-1})^l$ for $l\ge1$. We require the following result.

\goodbreak
\begin{lemma}[\!\cite{Colesanti-Ludwig-Mussnig-6}, Lemma 3.5 and Lemma 3.7]
\label{le:r_kln}
For $0\leq k \leq n$ and $0\leq l \leq n-k$, the map $\cR^l\colon\Had{k}{n}\to\Had{k}{n-l}$ is a bijection. Furthermore, 
$$\cR^l \zeta(s) = s^l \zeta(s) + l \int_s^{\infty} t^{l-1} \zeta(t) \d t$$
for every $\zeta\in \Had{k}{n}$ and $s>0$, while
\begin{equation*}
	\cR^{-l} \rho(s)=\frac{\rho(s)}{s^{l}} -l \int_s^{\infty} \frac{\rho(t)}{t^{l+1}} \d t
\end{equation*}
for every $\rho\in \Had{k}{n-l}$ and $s>0$.
\end{lemma}

For $t\geq 0$, let $u_t\in\fconvs$ be given by 
$$u_t(x):=t|x|+\ind_{\Bn}(x)$$ 
for $x\in\R^n$.
The next result shows that the transform $\cR$ naturally occurs when studying functional intrinsic volumes.

\begin{lemma}[\!\cite{Colesanti-Ludwig-Mussnig-5}, Lemma 2.15 and Lemma 3.24]
\label{le:calc_ind_bn_tx_theta_i}
	If $\,1\leq j \leq n$ and  $\zeta\in\Had{j}{n}$, then  
	$$\oZZ{j}{\zeta}(u_t)=\kappa_n \binom{n}{j} \cR^{n-j} \zeta(t)$$
	for $t\geq 0$.
\end{lemma}

We also require the dual form of the previous result. For $t\ge 0$, we set $v_t\!:=u_t^*$. Note that  
\begin{equation}\label{vt def}
	v_t(x)=\begin{cases}
	0\quad\quad&\mbox{for $\,\vert x\vert \le t$,}\\
	\vert x\vert-t\quad&\mbox{for $\,\vert x\vert>t$,}
	\end{cases}
\end{equation}
for $x\in\R^n$ and $t\ge0$ and that $v_t\in\fconvf$ for $t\ge 0$.

\begin{lemma}[\!\cite{Colesanti-Ludwig-Mussnig-5}]
	\label{le:ozzd_v_t}
	If $\,1\leq j \leq n$ and  $\zeta\in\Had{j}{n}$, then  
	$$\oZZd{j}{\zeta}(v_t)=\kappa_n \binom{n}{j} \cR^{n-j} \zeta(t)$$
	for $t\geq 0$.
\end{lemma}

\section{Monge--Amp\`ere and Mixed Monge--Amp\`ere Measures}\label{se:ma}
For $w\in\fconvx$, the \emph{subdifferential} of $w$ at $x\in\R^n$ is defined by
$$\partial w(x) := \{y\in\R^n\colon w(z)\geq w(x)+\langle y, z-x\rangle \text{ for all } z\in\R^n \}.$$
Each element of $\partial w(x)$ is called a \emph{subgradient} of $w$ at $x$. If $w$ is differentiable at $x$, then $\partial w(x)=\{\nabla w(x)\}$. 
Given a subset $A$ of $\R^n$, we define the image of $A$ through the subdifferential of $w$ as
$$
\partial w(A):=
\bigcup_{x\in A}\partial w(x).
$$
We write $\vert{\,\cdot\,}\vert$ for $n$-dimensional Lebesgue measure in $\R^n$ and remark that $\vert{\partial w(C)}\vert$ can be infinite for compact sets $C\subset \R^n$ and $w\in\fconvx$. An example is given by
$w\in\fconvx$ defined as
$$
w:=\ind_{\{0\}},
$$
as we have
$$
\partial w(\{0\})=\R^n.
$$
However, on $\fconvf$ we obtain a Radon measure, where a Borel measure $M$ is called a \emph{Radon measure} if $M(C)<+\infty$ for every compact set $C\subset\R^n$.  This is the content of the following result, which is due to Aleksandrov \cite{Aleksandrov} (see  \cite[Theorem 2.3]{Figalli} or \cite[Theorem 1.1.13]{Gutierrez}). Let $\Borel(\R^n)$ be the class of Borel sets in $\R^n$.

\begin{lemma}\label{lemma 0} Let $v\in\fconvf$. If $B\in\Borel(\R^n)$, then the set $\partial v(B)$ is measurable. Moreover, $\,\MA(v;\cdot)\colon \Borel(\R^n)\to [0,\infty]$, defined by
$$
\MA(v;B):=|\partial v(B)|,
$$
is a Radon measure on $\R^n$. 
\end{lemma}

\noindent
We will refer to $\MA(v;\cdot)$ as  the \emph{Monge--Amp\`ere measure} of $v$. The notion of Monge--Amp\`ere measure is fundamental in the definition of weak or generalized solutions of the Monge--Amp\`ere equation (see, for instance, \cite{Figalli, Gutierrez, Trudinger-Wang-MA}). 

The following statement gathers properties of  Monge--Amp\`ere measures. Items \ref{MA a} and \ref{MA b} are due to Aleksandrov \cite{Aleksandrov} (or see \cite[Proposition 2.6 and Theorem A.31]{Figalli}) while the valuation property  \ref{MA c} was deduced by Alesker \cite{Alesker_cf} from B\l ocki \cite{Blocki} (or see \cite[Theorem 9.2]{Colesanti-Ludwig-Mussnig-3}). Recall that for a sequence $M_k$  
of Radon measures in $\R^n$, we say that $M_k$ \emph{converges weakly} to a Radon measure $M$ in $\R^n$ if
\begin{equation*}\label{weak convergence}
\lim_{k\to+\infty}\int_{\R^n}\beta(x)\d M_k(x)=\int_{\R^n}\beta (x)\d M(x)
\end{equation*}
for every $\beta\in C_c(\R^n)$ (see, for instance, \cite{Evans-Gariepy}). 

\begin{theorem}\label{properties of MA measures} The following properties hold.
\begin{enumerate}[label=\emph{(\alph*)}, ref={(\alph*)}, leftmargin=18pt]
\item\label{MA a} If $v\in\fconvf$ and $v\in C^2(V)$ on an open set $V\subset \R^n$, then $\MA(v;\cdot)$ is absolutely continuous on $V$ with respect to $n$-dimensional Lebesgue measure and
\begin{equation*}\label{density}
\d\MA(v;x)=\det(\Hess v(x))\d x
\end{equation*}
for $x\in V$.
\item\label{MA b} If $v_j$ is a sequence in $\fconvf$ that is epi-convergent to $v\in\fconvf$, then the sequence of measures $\MA(v_j;\cdot)$ converges weakly to $\MA(v;\cdot)$.
\item\label{MA c} For every $v_1,v_2\in\fconvf$ such that $v_1\wedge v_2\in\fconvf$,
\begin{equation*}\label{MA_additivity}
\MA(v_1;\cdot)+\MA(v_2;\cdot)=\MA(v_1\wedge v_2;\cdot)+\MA(v_1\vee v_2;\cdot),
\end{equation*}
that is, $\MA$ is a (measure-valued) valuation on $\fconvf$.
\end{enumerate}
\end{theorem}

\goodbreak
Let $\Radon$ denote the space of Radon measures on $\R^n$.  According to Theorem \ref{properties of MA measures} (b), the map $\MA\colon \fconvf\to \Radon$ is continuous, when $\fconvf$ is equipped with the topology induced by epi-convergence and $\Radon$ with the topology induced by weak convergence. 

\subsection{Mixed Monge--Amp\`ere Measures}
We use polarization of the Monge--Amp\`ere measure with respect to the standard addition of functions to obtain mixed Monge--Amp\`ere measures. They were called mixed $n$-Hessian measures in \cite{Trudinger:Wang2002} and were used, for example,  in \cite{Passare-Rullgard}. 

We say that a map $\oZ\colon\left(\fconvf\right)^n\to\mathcal{M}(\R^n)$ is \emph{symmetric}, if the measure $\oZ(v_1,\dots,v_n;\cdot)$ is invariant with respect to every permutation of $n$-tuples of functions in  $\fconvf$. For $0\le j \le n$ and $v,v_1, \dots, v_{n-j}\in\fconvf$, we write $\oZ(v[j], v_1, \dots, v_{n-j}; \cdot)$ when the entry $v$ is repeated $j$ times.

\goodbreak
\begin{theorem}\label{thm mixed MA measures} 
There exists a symmetric map $\,\MA\colon\left(\fconvf\right)^n\to\Radon\,$ which assigns to every {$n$-tuple} of functions $v_1,\dots,v_n\in\fconvf$ a Radon measure $\MA(v_1,\dots,v_n;\cdot)$ with the following properties.
\begin{enumerate}[label=\emph{(\alph*)}, ref={(\alph*)}, leftmargin=18pt]
\item\label{MMA a} For every $m\in\N$, every $m$-tuple of functions $v_1,\dots,v_m\in\fconvf$, and $\lambda_1,\dots,\lambda_m\ge0$,
\begin{equation*}\label{multilinearity}
\MA(\lambda_1 v_1+\dots+\lambda_m v_m;\cdot)
=\sum_{i_1,\dots,i_n=1}^m \lambda_{i_1}\cdots\lambda_{i_n} \MA(v_{i_1},\dots, v_{i_n};\cdot).
\end{equation*}
\item \label{MMA b}For every $v\in\fconvf$,
$$
\MA(v,\dots,v;\cdot)=\MA(v;\cdot).
$$
\item \label{MMA c} If $v_1,\dots,v_n\in \fconvf$ and $v_1,\dots,v_n \in C^2(V)$ on an open set $V\subset\R^n$, then $\MA(v_1,\dots,v_n;\cdot)$ is absolutely continuous on $V$ with respect to $n$-dimensional Lebesgue measure  and
\begin{equation*}
\d\MA(v_1,\dots,v_n;x)=\det(\Hess v_1(x),\dots,\Hess v_n(x))\d x
\end{equation*}
for $x\in V$.
\item \label{MMA d} The map $\MA\colon\left(\fconvf\right)^n\to\mathcal{M}(\R^n)$ is continuous, when $\left(\fconvf\right)^n$ is equipped with the product topology and every factor has the topology induced by epi-convergence, while $\mathcal{M}(\R^n)$ is equipped with the topology induced by weak convergence. 
\item \label{MMA e} The map $\MA\colon\left(\fconvf\right)^n\to\mathcal{M}(\R^n)$ is dually epi-translation invariant with respect to every entry, that is,  
$$
\MA(v+\ell +\gamma,v_1,\dots,v_{n-1};\cdot)=\MA(v,v_1,\dots,v_{n-1};\cdot)
$$ 
for every $v,v_1,\dots,v_{n-1}\in\fconvf$, every linear function $\ell\colon\R^n\to\R$ and $\gamma\in\R$.
\item \label{MMA f} The map $\MA\colon\left(\fconvf\right)^n\to\mathcal{M}(\R^n)$ is additive and positively homogeneous of degree 1 with respect to every entry, that is, 
$$
\MA(\lambda v +\mu w,v_1,\dots,v_{n-1};\cdot)=\lambda\,\MA(v,v_1,\dots,v_{n-1};\cdot)+\mu\,\MA(w,v_1,\dots,v_{n-1}; \cdot)
$$
for every $v,w,v_1,\dots,v_{n-1}\in\fconvf$ and  $\lambda,\mu \geq 0$. 
\item \label{MMA g}For $0\le j\le n$ and $v_1,\dots,v_{n-j}\in\fconvf$, the map
$$v \mapsto \MA(v[j],v_1, \dots, v_{n-j}; \cdot)$$
is a (measure-valued) valuation on $\fconvf$.
\end{enumerate}
\end{theorem}

\begin{proof}
For $v_1,\dots,v_n\in \fconvf \cap C^2(\R^n)$ and $B\in\Borel(\R^n)$, we set
$$
\MA(v_1,\dots,v_n;B):=\int_B\det(\Hess v_1(x),\dots,\Hess v_n(x))\d x.
$$
Note that 
this measure is non-negative and symmetric. 
If all functions are in $\fconvf \cap C^2(\R^n)$, it verifies \ref{MMA a} by Theorem \ref{properties of MA measures} \ref{MA a} and by the fact that the mixed discriminant polarizes the determinant. Properties \ref{MMA b} and \ref{MMA f} for functions in $\fconvf \cap C^2(\R^n)$ also follow from corresponding properties of the mixed discriminant. Property~\ref{MMA e} for functions in $\fconvf \cap C^2(\R^n)$ follows directly from the fact that adding an affine function to $v$ does not change the Hessian matrix of $v$.

By \eqref{polarization_md}, we have 
\begin{equation}\label{PR}
\MA(v_1,\dots,v_n;\cdot)=\frac1{n!}\sum_{k=1}^n\ \sum_{1\le i_1<\dots<i_k\le n}(-1)^{n-k}\MA(v_{i_1}+\dots+v_{i_k};\cdot)
\end{equation}
for every $v_1,\dots,v_n\in \fconvf\cap C^2(\R^n)$. This identity, combined with Theorem \ref{properties of MA measures} \ref{MA b} and the denseness of $C^2(\R^n)$ functions in $\fconvf$, shows that the definition of $\MA$ extends continuously to $(\fconvf)^n$. Hence, we get properties \ref{MMA c} and \ref{MMA d}.
The extension inherits properties \ref{MMA a}, \ref{MMA b}, \ref{MMA e} and \ref{MMA f} by continuity. Property \ref{MMA g} follows from \eqref{PR} and  Theorem \ref{properties of MA measures} \ref{MA c}.
\end{proof}
\noindent
Concerning \ref{MMA c}, note that we use different notation for the Lebesgue integral over the mixed discriminant, which we only consider for functions that are of class $C^2$ almost everywhere, and the mixed Monge--Amp\`ere measure. This distinction is not always made in the literature.

As a consequence of Theorem~\ref{thm mixed MA measures} we obtain the following result, which for the special case $j=n$ was previously established in \cite[Proposition 19]{Colesanti-Ludwig-Mussnig-4}.

\begin{proposition}
\label{prop:mixed_hessian_vals}
Let $\beta\in C_c(\R^n)$ and $0\le j \le n$. If $v_{1},\dots,v_{n-j}\in\fconvf$, then the map $\oZ\colon\fconvf\to\R$ defined by
\begin{equation}\label{mhv}
\oZ(v):=\int_{\R^n}\beta(x)\d\MA(v[j],v_{1},\dots,v_{n-j};x),
\end{equation}
is a continuous, dually epi-translation invariant valuation that is homogeneous of degree $j$.  
\end{proposition}

\begin{proof} Note that the integral in \eqref{mhv} is well-defined and finite as $\beta\in C_c(\R^n)$ and mixed Monge--Amp\`ere measures are Radon measures. Continuity follows from the weak continuity of mixed Monge--Amp\`ere measures. The invariance, homogeneity and valuation properties are consequences of items \ref{MMA e}, \ref{MMA f} and \ref{MMA g} of Theorem \ref{thm mixed MA measures}, respectively.  
\end{proof}

\medskip
\noindent
We remark that valuations defined in a way similar to \eqref{mhv} have been considered by Alesker \cite{Alesker_cf} and by Knoerr  \cite{Knoerr2}.

\subsection{Hessian Measures as a Special Case} 
For a function $v\in \fconvf\cap C^2(\R^n)$ and $0\le j\le n$, the $j$th Hessian measure $\hd{j}(v;B)$ is defined for $B\in\Borel(\R^n)$ as 
$$\hd{j}(v;B)=\int_B  [\Hess v(x)]_j \d x.$$ 
Trudinger and Wang \cite{Trudinger:Wang1997, Trudinger:Wang1999} showed that $\hd{j}(v;\cdot)$ can be extended to a Radon measure on $\fconvf$. It coincides up to the factor $\binom{n}{j}$ with
$$
\MA(v[j],\tfrac12{h_{\Bn}^2}[n-j];\cdot).
$$
Indeed, if $v\in \fconvf\cap C^2(\R^n)$, then by Theorem \ref{thm mixed MA measures} \ref{MMA c} and \eqref{eq:mixed_dis_hessian},
$$
\binom{n}{j}\d\MA(v[j],\tfrac12{h_{\Bn}^2}[n-j];x)= \binom{n}{j}\det(\Hess v(x)[j],I_n[n-j])\d x= [\Hess v(x)]_j \d x = \!\d\Phi_j(v,x).
$$
We obtain the conclusion using the denseness of smooth functions.

\subsection{A Special Family of Mixed Monge--Amp\`ere Measures}

We introduce the following family of mixed Monge--Amp\`ere measures.  For $v\in\fconvf$ and $0\le j \le n$, we set
\begin{equation}\label{def nhd}
\nhd{j}(v;\cdot):=\MA(v[j],h_{\Bn}[n-j];\cdot).
\end{equation}
By construction, this is a Radon measure on $\R^n$. 

It follows from Theorem \ref{thm mixed MA measures} \ref{MMA a} that the mixed Monge--Amp\`ere measures \eqref{def nhd} can also be obtained as coefficients of the following Steiner formula,
\begin{equation}
\label{eq:steiner_formula_maj}
\MA(v+r h_{\Bn};B)=\sum_{j=0}^n \binom nj r^{n-j}\nhd{j}(v;B)
\end{equation}
for $B\in\Borel(\R^n)$ and $r\ge 0$.

We derive some of the properties  of the measures $\nhd{j}(v;\cdot)$ for $0\le j\le n$.
The subdifferential of $h_{\Bn}$ can be explicitly described as 
\begin{equation}\label{gradhB}
\partial h_{\Bn}(x)=
\begin{cases}
\Big\{\dfrac{x}{|x|}\Big\}&\mbox{for $x\ne 0$,}\\[12pt]
\,\Bn&\mbox{for $x=0$.}
\end{cases}
\end{equation}
Combining this with the definition of Monge--Amp\`ere measure, we see that
\begin{equation}\label{MA of the norm}
\MA(h_{\Bn};\cdot)=\kappa_n \delta_0,
\end{equation}
where  $\delta_0$ is the Dirac measure at $0$.
Indeed, if $B\in\Borel(\R^n)$ does not contain the origin, then we have $\partial h_{\Bn}(B)\subset\S^{n-1}$, so that
$$
\MA (h_{\Bn};B)=\vert{\partial h_{\Bn}(B)}\vert=0.
$$
On the other hand, if $0\in B$, then $\partial h_{\Bn}(B)=\Bn$. Note that
\begin{equation}\label{Jn}
\Hess h_{\Bn}(x)=\frac1{\vert x\vert} \Big(I_n-\frac{x}{\vert x\vert}\otimes \frac{x}{\vert x\vert}\Big)
\end{equation}
for $x\ne0$, where $y\otimes z$ denotes the tensor product of $y,z\in\R^n$.

\goodbreak
\begin{theorem}\label{properties of nhm} Let $v\in\fconvf$. The following properties hold. 
\begin{enumerate}[label=\emph{(\alph*)}, ref={(\alph*)}, leftmargin=18pt]
\item\label{nhm a}For $B\in\Borel(\R^n)$,
$$
\nhd{0}(v;B)=\kappa_n\delta_0(B).
$$
In particular, $\,\nhd{0}(v;\cdot)$ is independent of $v$.
\item\label{nhm b} For $B\in\Borel(\R^n)$,
$$
\nhd{n}(v;B)=\vert{\partial v(B)}\vert=\MA(v;B).
$$
\item \label{nhm 0} For $\,0\le j\le n$, $$\nhd{j}(v;\{0\})=\frac{\kappa_{n-j}}{\binom{n}{j}}\, V_{j}(\partial v(0)).$$
\item\label{nhm c} If $v\in C^2(V)$ with $V\subset\R^n$ open and $1\le j\le n$, then $\nhd{j}(v;\cdot)$ is absolutely continuous on $V\backslash\{0\}$ with respect to $n$-dimensional Lebesgue measure and 
$$
\d\nhd{j}(v;x)=\det(\Hess v(x)[j],\Hess h_{\Bn}(x) [n-j])\d x
$$
for $x\in V$ with $x\ne 0$.
\item\label{nhm d} For $\,0\le j\le n$, the map $\nhd{j}: \fconvf\to \Radon$ is a continuous valuation.
\end{enumerate}
\end{theorem}

\begin{proof} Item \ref{nhm a} follows from  \eqref{MA of the norm}. Item \ref{nhm b} follows from \eqref{def nhd} with $j=n$ and the definition of $\MA(v;\cdot)$. Item \ref{nhm d} is a consequence of Theorem \ref{thm mixed MA measures} \ref{MMA d} and \ref{MMA g} while item \ref{nhm c} follows from Theorem \ref{thm mixed MA measures} \ref{MMA c} combined with \eqref{Jn} and \eqref{polarisation md}.

It remains to show \ref{nhm 0}. By \eqref{gradhB} we obtain that
for $v\in\fconvf$ and  $r\ge0$,
\begin{equation*}\label{parallel zero}
\partial(v+r\,h_{\Bn})(0)=\partial v(0)+r \Bn,
\end{equation*}
where we use that the subdifferential of the sum of two convex functions is the Minkowski sum of their subdifferentials (see, for example, \cite[Theorem 23.8]{Rockafellar}). Hence, according to the Steiner formula \eqref{steiner},
\begin{equation*}\label{deltas}
\vert{\partial(v+r\,h_{\Bn})(0)}\vert=\sum_{j=0}^n r^{n-j}\kappa_{n-j} V_j(\partial v(0)),
\end{equation*}
which  combined with the definition of the Monge--Amp\`ere measure and \eqref{eq:steiner_formula_maj} concludes the proof. 
\end{proof}

\section{Conjugate Monge--Amp\`ere and Conjugate Mixed Monge--Amp\`ere Measures}
\label{se:conjugate_mma_measures}
First, we define the conjugate Monge--Amp\`ere measure for super-coercive convex functions, using the construction of Monge--Amp\`ere measures on $\fconvf$ and a duality argument.
Let $u\in\fconvs$. For $B\in\Borel(\R^n)$, we set
\begin{equation}\label{MA measure primal}
\MAp(u;B):=\MA(u^*;B).
\end{equation}
Note that Lemma \ref{lemma 0} implies that  $\MAp\colon\fconvs\to \Radon$ is well-defined and that $\MAp(u;\cdot)$ is a Radon measure for every $u\in\fconvs$, as $u^*\in\fconvf$ in this case. We refer to $\MAp(u; \cdot)$ as the \emph{conjugate Monge--Amp\`ere measure} of $u$.
It is the push-forward through  $\nabla u$  of  $n$-dimensional Lebesgue measure restricted to the domain of $u$  and we have included a proof of this known fact as item \ref{cMA a} of the following result.
In the following, for $u\in \fconvs$, we use the relation
\begin{equation}
\label{eq:inv_hess}
\Hess u^*(\nabla u(x))=\left(\Hess u(x)\right)^{-1}
\end{equation}
for $x\in\R^n$ such that $u\in C_+^2(U)$ in a neighborhood $U$ of $x$ (see \cite[p. 605]{RockafellarWets}).

\goodbreak
\begin{theorem}\label{properties of MA measure in the primal setting} 
The following properties hold.
\begin{enumerate}[label=\emph{(\alph*)}, ref={(\alph*)}, leftmargin=18pt]
\item\label{cMA a} If $u\in\fconvs$, then
\begin{equation*}
\int_{\R^n}\beta(y)\d\MAp(u;y)=\int_{\dom u}\beta(\nabla u(x)) \d x
\end{equation*}
for every $\beta\in C_c(\R^n)$.
\item\label{cMA b} If $u_j$ is a sequence in $\fconvs$ that is epi-convergent to $u\in\fconvs$, then the sequence of measures $\MAp(u_j;\cdot)$ converges weakly to $\MAp(u;\cdot)$.
\item\label{cMA c} For every $u_1,u_2\in\fconvs$ such that $u_1\vee u_2$ and $u_1\wedge u_2$ are in $\fconvs$,
\begin{equation*}
\MAp(u_1;\cdot)+\MAp(u_2;\cdot)=\MAp(u_1\wedge u_2;\cdot)+\MAp(u_1\vee u_2;\cdot),
\end{equation*}
that is, $\MAp$ is a (measure-valued) valuation on $\fconvs$.
\end{enumerate}
\end{theorem}

\begin{proof} Properties \ref{cMA b} and \ref{cMA c} are consequences of properties \ref{MA b} and \ref{MA c} in Theorem \ref{properties of MA measures}, respectively, and of Proposition \ref{basics}. 

Concerning property \ref{cMA a}, observe that if $u\in \fconvs\cap C^2_+(\R^n)$, then $u^*\in \fconvf\cap C^2_+(\R^n)$ (by Proposition~\ref{basics}). By Theorem \ref{properties of MA measures} \ref{MA a},  setting $v\!:=u^*$, we obtain
\begin{eqnarray*}
\int_{\R^n}\beta(y)\d\MAp(u;y)&=&\int_{\R^n}\beta(y)\d\MA(v;y)\\
&=&\int_{\R^n}\beta(y)\det(\Hess v(y))\d y\\
&=&\int_{\R^n}\beta(\nabla u(x))\d x
\end{eqnarray*}
for $\beta\in C_c(\R^n)$.
Here we used the change of variable $y=\nabla u(x)$ and \eqref{eq:inv_hess}. The statement now follows from property \ref{cMA b} combined with the fact that 
the functional $u\mapsto \int_{\R^n}\beta(\nabla u(x))\d x$
is continuous on $\fconvs$ (see \eqref{eq:n-hom}).
\end{proof}

\subsection{Conjugate Mixed Monge--Amp\`ere Measures}
We use polarization of the conjugate Monge--Amp\`ere measure with respect to infimal convolution to define conjugate mixed Monge--Amp\`ere measures.
The following result is easily obtained from Theorem \ref{thm mixed MA measures}.

\begin{theorem}\label{thm mixed MA measures primal} There exists a symmetric map $\,\MAp\colon\left(\fconvs\right)^n\to\Radon\,$ which assigns to every $n$-tuple of functions $u_1,\dots,u_n\in\fconvs$ a Radon measure $\MAp(u_1,\dots,u_n;\cdot)$ with the following properties.
\begin{enumerate}[label=\emph{(\alph*)}, ref={(\alph*)}, leftmargin=18pt]
\item\label{cMM a} For every $m\in\N$, every $m$-tuple of functions $u_1,\dots,u_m\in\fconvs$ and $\lambda_1,\dots,\lambda_m\ge0$,
\begin{equation*}
\MAp(\lambda_1\sq u_1\infconv \dots\infconv \lambda_m\sq u_m;\cdot)=\sum_{i_1,\dots,i_n=1}^m \lambda_{i_1}\cdots\lambda_{i_n} \,\MAp(u_{i_1},\dots, u_{i_n};\cdot).
\end{equation*}
\item\label{cMM b} For every $u\in\fconvs$,
$$
\MAp(u,\dots,u;\cdot)=\MAp(u;\cdot).
$$
\item\label{cMM c} If $\,u_1,\dots,u_n\in \fconvs$ and $u_1^*,\dots,u_n^*\in C^2(V)$ on an open set $V\subset \R^n$, then the measure $\MAp(u_1,\dots,u_n;\cdot)$ is absolutely continuous on $V$ with respect to $n$-dimensional Lebesgue measure  and 
\begin{equation*} \d\MAp(u_1,\dots,u_n;x)=\det(\Hess u_1^*(x),\dots,\Hess u_n^*(x))\d x\end{equation*}
for $x\in V$.
\item\label{cMM d} The map $\MAp\colon\left(\fconvs\right)^n\to\Radon$ is continuous, when $\left(\fconvs\right)^n$ is equipped with the product topology and every factor has the topology induced by epi-convergence, while $\mathcal{M}(\R^n)$ is equipped with the topology induced by weak convergence. 
\item\label{cMM e} The map $\MAp\colon\left(\fconvs\right)^n\to\Radon$ is epi-translation invariant with respect to every entry, that is,  
$$
\MAp(u\circ \tau^{-1} +\gamma,u_1,\dots,u_{n-1};\cdot)=\MAp(u,u_1,\dots,u_{n-1};\cdot)$$
for every $u,u_1,\dots,u_{n-1}\in\fconvs$, every translation $\tau\colon \R^n\to\R^n$ and $\gamma\in\R$. 
\item\label{cMM f} The map $\MAp\colon\left(\fconvs\right)^n\to\Radon$ is epi-additive and epi-homogeneous of degree 1 with respect to every entry, that is,
$$
\MAp((\lambda\sq u)\infconv  (\mu\sq w),u_1,\dots,u_{n-1};\cdot)=\lambda\,\MAp(u,u_1,\dots,u_{n-1};\cdot)+\mu\,\MAp(w,u_1,\dots,u_{n-1}; \cdot),
$$
for every $u,w,u_1,\dots,u_{n-1}\in\fconvs$ and for every $\lambda,\mu \geq 0$. 
\item\label{cMM g} For $\,0\le j\le n$ and $u_1, \dots, u_{n-j}\in\fconvs$, the map 
$$u\mapsto \MAp(u[j],u_1, \dots,u_{n-j};\cdot)$$ 
is a (measure-valued) valuation on $\fconvs$.
\end{enumerate}
\end{theorem}

\noindent
Here, for \ref{cMM a} and \ref{cMM f}, we extend the definition of epi-multiplication to $0\sq u = \ind_{\{0\}}$ for $u\in\fconvs$.

\goodbreak
The dual version of Proposition \ref{prop:mixed_hessian_vals} is the following result.

\begin{proposition}
\label{prop:mixed_hessian_vals_primal}
Let $\beta\in C_c(\R^n)$ and $0\le j\le n$. If  $u_{1},\dots,u_{n-j}\in\fconvs$, then the map $\oZ\colon\fconvs\to\R$, defined by
\begin{equation}\label{mhvp}
\oZ(u):=\int_{\R^n}\beta(x) \d\MAp(u[j],u_1\dots,u_{n-j};x),
\end{equation}
is a continuous, epi-translation invariant valuation, that is epi-homogeneous of degree $j$.  
\end{proposition}

\subsection{Connections to Hessian Measures}
\label{ss:hessian_measures_fconvs}
For  $u\in\fconvs\cap C_+^2(\R^n)$ and $0\le j\le n$, define the Hessian measure $\hp{j}(u; \cdot)$ as the push-forward through $\nabla u$ of the Hessian measure $\hd{n-j}(u;\cdot)$ of $u$, that is, 
$$\int_{\R^n} \beta(y)\d \hp{j}(u;y)=\int_{\R^n} \beta(\nabla u(x))\,[\Hess u(x)]_{n-j}\d x$$
for every Borel function $\beta\colon \R^n\to [0,\infty)$.  We remark that  the measure $\hp{j}(u;\cdot)$ can be defined for every $u\in\fconvs$ and is a marginal of a generalized Hessian measure (see \cite{Colesanti-Ludwig-Mussnig-3}). Moreover, for $u\in\fconvs \cap C_+^2(\R^n)$ we obtain from \eqref{MA measure primal} and Theorem \ref{thm mixed MA measures} \ref{MMA c} that
\begin{eqnarray*}
\int_{\R^n} \beta(y) \d\MAp(u[j], \tfrac12 h_B^2[n-j];y)&=&\int_{\R^n} \beta(y) \d\MA(u^*[j], \tfrac12 h_B^2[n-j];y)\\
&=& \int_{\R^n} \beta(y) \det(\Hess u^*[j], I_n[n-j])\d y\\
 &=&\frac1{\binom{n}{j}}\int_{\R^n} \beta(y) [\Hess u^*(y)]_{j}\d y\\
 &=&\frac1{\binom{n}{j}}\int_{\R^n} \beta(\nabla u(x)) [\Hess u(x)]_{n-j}\d x,
\end{eqnarray*}
where for the last step we used \eqref{eq:inv_hess} and \eqref{eq:jtrace}.
Hence, the measure
$$
\MAp(u[j],\tfrac12{h_{\Bn}^2}[n-j];\cdot),
$$
coincides up to the factor $\binom{n}{j}$ with $\hp{j}(u;\cdot)$ for $u \in C_+^2(\R^n)$.  The corresponding statement holds for general $u\in\fconvs$ by the denseness of smooth functions and the weak continuity of Hessian and conjugate mixed Monge--Amp\`ere measures.

\subsection{A Special Family of Conjugate Mixed Monge--Amp\`ere Measures}

We introduce the following family of conjugate mixed Monge--Amp\`ere measures. For $u\in\fconvs$ and $0\le j \le n$, we set
\begin{equation*}\label{def nhm}
\nhp{j}(u;\cdot):=\MAp(u[j],\ind_{\Bn}[n-j];\cdot).
\end{equation*}
By construction, this is a Radon measures on $\R^n$. 
A consequence of this definition is that
\begin{equation*}\label{def nhp}
\nhp{j}(u;\cdot)=\nhd{j}(u^*;\cdot)
\end{equation*}
for $u\in\fconvs$. It follows from Theorem \ref{thm mixed MA measures primal} \ref{cMM a} that this family of conjugate mixed Monge--Amp\`ere measures can also be obtained as coefficients of the following Steiner formula,
\begin{equation}
\label{eq:steiner_formula_maj_primal}
\MAp(u\infconv (r\sq \ind_{\Bn});B)=\sum_{j=0}^n \binom nj r^{n-j}\nhp{j}(u;B)
\end{equation}
for $B\in\Borel(\R^n)$ and $r\ge 0$.

\goodbreak
The next result describes properties of this family of conjugate Monge--Amp\`ere measures.
\begin{theorem}\label{properties of nhp} 
Let $u\in\fconvs$. The following statements hold.
\begin{enumerate}[label=\emph{(\alph*)}, ref={(\alph*)}, leftmargin=18pt]
\item\label{nhp a}For every $B\in\Borel(\R^n)$,
$$
\nhp{0}(u;B)=\kappa_n\delta_0(B).
$$
In particular, $\,\nhp{0}(u;\cdot)$ is independent of $u$.
\item\label{nhp b} For every $B\in\Borel(\R^n)$,
$$
\nhp{n}(u;B)=\vert{\partial u^*(B)}\vert=\MAp(u;B).
$$
\item \label{nhp 0} For $\,0\le j\le n$, 
$$\nhp{j}(u;\{0\})=\frac{\kappa_{n-j}}{\binom{n}{j}}\, V_{j}(\argmin u).$$
\item\label{nhp c} If $u^*\in C^2_+(V)$ with $V\subset \R^n$ open and $1\le j\le n$, then $\nhp{j}(u;\cdot)$ is absolutely continuous on $V$ with respect to  $n$-dimensional Lebesgue measure  and 
$$
\d\nhp{j}(u;x)=\det(\Hess u^*(x)[j],\Hess h_{\Bn}(x) [n-j])\d x
$$
for $x\in V$ with $x\ne 0$.
\item\label{nhp d} For $\,0\le j\le n$, the map $\nhp{j}\colon \fconvs\to \Radon$ is a continuous valuation.
\end{enumerate}
\end{theorem}

\begin{proof}
The statements follow from the corresponding statements in Theorem \ref{properties of nhm} by duality. For \ref{nhp 0}, we use that $\partial u^*(0)=\argmin u$ (see \cite[Theorem 11.8]{RockafellarWets}).
\end{proof}

In the following, for  $u\in\fconvs$, we will use the fact that $y\in\partial u(x)$ if and only if $x\in\partial u^*(y)$ (see, for example, \cite[Theorem 23.5]{Rockafellar}). Combined with \eqref{eq:inv_hess}, it implies that  $u\in C_+^2(U)$ for some open set $U\subset \R^n$ if and only if $u^*\in C_+^2(V)$ with
$V:=\{\nabla u(x) \colon x\in U\}$.
In particular, we obtain that the set $V$ is open and $\nabla u:U\to V$ is a bijection.

For the following result, we recall that  $\tau_i(u,x)$ is the $i$th elementary symmetric function of the principal curvatures of the sublevel set of $u$ passing through $x$ for $x\notin \argmin u$.

\begin{theorem}\label{representation of Steiner measures primal smooth} Let $u\in\fconvs$ and $1\le j\le n-1$.  If $u\in C_+^2(U)$ for an open set $U\subset \R^n$ and $V:=\{\nabla u(x) \colon x\in U\}$, then
\begin{equation}\label{bruum}
\nhp{j}(u;B)=\frac1{\binom nj}\int_{(\nabla u)^{-1}(B)}\tau_{n-j}(u,x)\d x
\end{equation}
for every Borel set $B\subset V\backslash\{0\}$. 
Equivalently, 
\begin{equation}\label{bruuum}
\int_{\R^n}\beta(y)\d\nhp{j}(u;y)=\frac1{\binom nj}\int_{\R^n}\beta(\nabla u(x))\,\tau_{n-j}(u,x)\d x
\end{equation}
for every $\beta\in C_c(V)$ .
\end{theorem}

\goodbreak
For the proof we need the following result. 
\begin{lemma}\label{Jacobian lemma} Let $u\in\fconvs$ be such that $u\in C_+^2(U)$ for an open set $U \subset \R^n$ and let $r>0$. If $T_r\colon U\backslash\argmin u\to\R^n$ is defined by
$$
T_r(x):=x+r\,\frac{\nabla u(x)}{|\nabla u(x)|},
$$
then, for the Jacobi matrix $\D T_r$, 
we have
$$
\det(\D T_r(x))=
\sum_{j=0}^{n-1}r^j\tau_j(u,x)
$$
for every $x\in U\backslash\argmin u$.
\end{lemma}

\begin{proof} Let $x\in U\backslash\argmin u$. Clearly, 
$$
\D T_r(x)=I_n+r\,\D N(x)
$$
where $N(x)=(N_1(x),\dots,N_n(x))$ is defined as
$$
N(x):=\frac{\nabla u(x)}{|\nabla u(x)|}.
$$
Let $t\!:=u(x)$. We choose a coordinate system such that
\begin{equation}\label{direction}
\nabla u(x)=\lambda e_n=\lambda \nu_t(x),
\end{equation}
where $\nu_t(x)$ denotes the outer unit normal to $\{u\le t\}$ at $x$ and $\lambda=|\nabla u(x)|>0$. We may also assume that, for $1\le j\le n-1$, the vector $e_j$ is a direction of principal curvature for $\partial \{u\le t\}$ at $x$ with corresponding principal curvature $\kappa_j(u,x)$. As $N$ is an extension of $\nu_t$, we obtain
$$
\D N(x)=\left(
\begin{array}{cccccc}
\kappa_1(u,x)&0&\cdots&0&\frac{\partial N_1}{\partial x_n}(x)\\[2pt]
0&\kappa_2(u,x)&\cdots&0&\frac{\partial N_2 }{\partial x_n}(x)\\
\vdots&\vdots&\ddots&\vdots&\vdots\\
0&0&\cdots&\kappa_{n-1}(u,x)&\frac{\partial N_{n-1}}{\partial x_n}(x)\\[2pt]
\frac{\partial N_n}{\partial x_1}(x)&\frac{\partial N_n}{\partial x_2}(x)&\cdots&\frac{\partial N_n}{\partial x_{n-1}}(x)&\frac{\partial N_n}{\partial x_n}(x)
\end{array}
\right).
$$

On the other hand, using \eqref{direction}, we obtain
\begin{eqnarray*}
\frac{\partial N_n}{\partial x_j}(x)&=&\frac{\partial}{\partial x_j}\left(\frac {1}{|\nabla u(x)|} \frac {\partial u(x)}{\partial x_n}\right)\\
&=&\frac{1}{|\nabla u(x)|}\,\frac{\partial^2 u(x)}{\partial x_j\partial x_n}-\frac{1}{|\nabla u(x)|^2} \frac{\partial u(x)}{\partial x_n} \frac{\partial}{\partial x_j} |\nabla u(x)|\\[1pt]
&=&\frac{1}{|\nabla u(x)|}\,\frac{\partial^2 u(x)}{\partial x_j\partial x_n}-\frac{1}{|\nabla u(x)|^3} \frac{\partial u(x)}{\partial x_n} \sum_{i=1}^n \frac{\partial u(x)}{\partial x_i} \frac{\partial^2 u(x)}{\partial x_i\partial x_j}\\
&=&0
\end{eqnarray*}
for $1\le j\le n$. Here, for the last equality, we used that $\frac{\partial u(x)}{\partial x_i}=0$ for all $1\leq i \leq n-1$ because of the choice of our coordinate system.
Therefore, 
$$
\D T_r(x)=\left(
\begin{array}{cccccc}
1+r\kappa_1(u,x)&0&\cdots&0&\frac{\partial N_1}{\partial x_n}(x)\\[2pt]
0&1+r\kappa_2(u,x)&\cdots&0&\frac{\partial N_2}{\partial x_n}(x)\\
\vdots&\vdots&\ddots&\vdots&\vdots\\
0&0&\cdots&1+r\kappa_{n-1}(u,x)&\frac{\partial N_{n-1}}{\partial x_n}(x)\\
0&0&\cdots&0&1
\end{array}
\right)
$$
and
$$\det(\D T_r(x))= \prod_{i=1}^{n-1}(1+r\,\kappa_i(u,x)),$$
which implies the representation formula.
\end{proof}

\goodbreak
\begin{proof}[Proof of Theorem \ref{representation of Steiner measures primal smooth}]
Formula \eqref{bruuum} directly follows from \eqref{bruum}. So, we have to prove \eqref{bruum}.

Let $B\subset V\backslash\{0\}$ be a Borel set. For $r>0$, let $T_r\colon U\backslash \argmin u$ be the map defined in Lemma~\ref{Jacobian lemma}. Note that $U\backslash \argmin u=\nabla u^{-1}(V\backslash\{0\})$. We have
\begin{eqnarray*}
\vert \partial (u^* +r\,h_{\Bn})(B)\vert
&=&\Big\vert{\big\{\nabla u^*(y)+r\frac{y}{|y|}\colon y\in B\big\}}\Big\vert\\
&=&\Big\vert{\big\{x+r\frac{\nabla u(x)}{|\nabla u(x)|}\colon x\in (\nabla u)^{-1}( B)\big\}}\Big\vert\\
&=&\int_{(\nabla u)^{-1}(B)}\det(\D T_r(x))\d x\\
&=&\sum_{j=0}^{n-1}r^j\int_{(\nabla u)^{-1}(B)}\tau_j(u,x)\d x,
\end{eqnarray*}
where we have used Lemma \ref{Jacobian lemma}. On the other hand,  by the definition of the Monge--Amp\`ere measure and of the conjugate Monge--Amp\`ere measure and \eqref{eq:steiner_formula_maj_primal}, we have
\begin{equation*}
\vert{\partial (u^* +r\,h_{\Bn})(B)}\vert= \MAp(u\infconv (r \sq \ind_{\Bn}); B)=\sum_{j=1}^n \binom nj r^{n-j}\nhp{j}(u;B).
\end{equation*}
The conclusion follows from comparing coefficients. 
\end{proof}

\section{Connecting $\nhd{j}(v; \cdot)$ and Hessian Measures}
\label{se:steiner}
The purpose of this section is to prove Proposition~\ref{prop:transform_ints}, which shows how integrals of radially symmetric functions with respect to Hessian measures can be written in terms of integrals with respect to the new family of mixed Monge--Amp\`ere measures.
This result is essential for our new proof of the existence of functional intrinsic volumes, Theorem~\ref{thm:existence_singular_hessian_vals} and Theorem~\ref{thm:existence_singular_hessian_vals_dual}, as well as for the proof of the new representations, Theorem~\ref{thm:steiner_measures} and Theorem~\ref{thm:steiner_measures_dual}.

\subsection{Reilly-Type Lemmas}
We will need the following result by Reilly \cite[Proposition 2.1]{Reilly}  (or see \cite[(2.10)]{Trudinger:Wang2002}).

\begin{lemma}[Reilly]
\label{le:sum_diff_c_ij_zero}
If $v_1,\ldots,v_{n-1}\in C^3(\R^n)$ and $1\le j\le n$, then
$$\sum_{i=1}^n \frac{\partial}{\partial x_i} D_{ij}(\Hess v_1(x),\ldots, \Hess v_{n-1}(x))=0$$
for every $x\in\R^n$. 
\end{lemma}

The following result shows that
$$(v_0,\ldots,v_n)\mapsto \int_{\R^n} v_0(x) \det(\Hess v_1(x),\ldots,\Hess v_n(x))\d x$$
with $v_0,\ldots,v_n\in C^2(\R^n)$
is symmetric in its entries if at least one the functions has compact support. We remark that this corresponds to the symmetry of mixed volumes in the following representation,
$$V(K_1,\ldots,K_n) = \int_{\sn} h_{K_1}(y) \det(\HessTilde h_{K_2}(y),\ldots,\HessTilde h_{K_n}(y))\d y,$$
for sufficiently smooth $K_1,\ldots,K_n\in \cK^n$, where $\HessTilde h_K(y)$ is the restriction of $\Hess h_K$ to the tangent space of $\sn$ at $y\in\sn$ (see for example equations (2.68) and (5.19) in \cite{Schneider:CB2}).

\begin{lemma}
\label{le:change_int_mixed_dis}
If $v_0,\ldots,v_n\in C^2(\R^n)$ are such that at least one of the functions has compact support, then
\begin{equation}
\label{eq:change_int_mixed_dis}
\int\limits_{\R^n} v_0(x)  \det(\Hess v_1(x),\ldots,\Hess v_n(x))\d x = \int\limits_{\R^n} v_n(x) \det(\Hess v_1(x),\ldots,\Hess v_{n-1}(x),\Hess v_0(x))\d x.
\end{equation}
\end{lemma}
\begin{proof}
Assume first that $v_0,\ldots,v_n\in C^3(\R^n)$. In this case, $D_{ij}(\Hess v_2(x),\ldots,\Hess v_n(x))$ is differentiable and therefore
\begin{align*}
\frac{\partial}{\partial x_i} \Big(\sum_{j=1}^n D_{ij}(\Hess v_1(x),\ldots,\Hess v_{n-1}(x))\, \frac{\partial v_n(x)}{\partial x_j}  \Big) &= \sum_{j=1}^n \frac{\partial}{\partial x_i} D_{ij}(\Hess v_1(x),\ldots,\Hess v_{n-1}(x)) \,\frac{\partial v_n(x)}{\partial x_j} \\
&\quad + \sum_{j=1}^n D_{ij}(\Hess v_1(x),\ldots,\Hess v_{n-1}(x)) \, \frac{\partial^2 v_n(x)}{\partial x_i \partial x_j}
\end{align*}
for $1\le i\le n$ and  $x\in\R^n$. Summation over $i$ combined with Lemma~\ref{le:sum_diff_c_ij_zero} now gives
\begin{equation*}
\sum_{i,j=1}^n \frac{\partial}{\partial x_i}\Big(D_{ij}(\Hess v_1(x),\ldots,\Hess v_{n-1}(x))\, \frac{\partial v_n(x)}{\partial x_j} \Big) = \sum_{i,j=1}^n D_{ij}(\Hess v_1(x),\ldots,\Hess v_{n-1}(x))\,\frac{\partial^2 v_n(x)}{\partial x_i \partial x_j} 
\end{equation*}
for $x\in\R^n$. By the definition of $D_{ij}$ and using that at least one of the functions $v_0,\ldots,v_n$ has compact support, we now obtain from the divergence theorem that
\begin{align*}
\int_{\R^n} v_0(x)  \det(\Hess v_1(x),\ldots,& \Hess v_n(x)) \d x\\
&= \int_{\R^n} v_0(x) \sum_{i,j=1}^n D_{ij}(\Hess v_1(x),\ldots,\Hess v_{n-1}(x))\, \frac{\partial^2 v_n(x)}{\partial x_i \partial x_j}  \d x\\
&= \int_{\R^n} v_0(x) \sum_{i,j=1}^n \frac{\partial}{\partial x_i} \Big( D_{ij}(\Hess v_1(x),\ldots,\Hess v_{n-1}(x)) \,\frac{\partial v_n(x)}{\partial x_j} \Big)\d x\\
&= - \int_{\R^n} \sum_{i,j=1}^n D_{ij}(\Hess v_1(x),\ldots,\Hess v_{n-1}(x))\, \frac{\partial v_n(x)}{\partial x_j} \, \frac{\partial v_0(x)}{\partial x_i}  \d x.
\end{align*}
Since the last expression is symmetric in $v_0$ and $v_n$, we may exchange the two functions. This completes the proof under the additional assumption that all functions are in $C^3(\R^n)$.

It remains to show that the result holds true on $C^2(\R^n)$. By the multilinearity of the mixed discriminant combined with the assumption that one of the functions $v_0,\ldots,v_n$ has compact support, there exists $r>0$ such that the integrands in \eqref{eq:change_int_mixed_dis} vanish outside of $r \Bn$. The result now easily follows by a standard approximation argument combined with the dominated convergence theorem.
\end{proof}

A further consequence of Lemma~\ref{le:sum_diff_c_ij_zero} is the following result.

\begin{lemma}
\label{le:int_mixed_dis_zero}
Let $v_1,\ldots,v_{n-1}\in C^2(\R^n)$ and let $F\colon\R^n\to\R^n$ be a continuously differentiable vector field. If $F$ has compact support, then
$$\int_{\R^n} \sum_{i,j=1}^n D_{ij}(\Hess v_1(x),\ldots,\Hess v_{n-1}(x)) \,\frac{\partial F_i(x)}{\partial x_j} \d x=0.$$
\end{lemma}

\begin{proof}
Assume first that $v_1,\ldots,v_{n-1}\in C^3(\R^n)$. Using the definition of $D_{ij}$ and that $F$ has compact support, we obtain from the divergence theorem that
\begin{multline*}
\int_{\R^n} \sum_{i,j=1}^n D_{ij}(\Hess v_1(x),\ldots,\Hess v_{n-1}(x))\, \frac{\partial F_i(x)}{\partial x_j} \d x\\
= -\int_{\R^n} \sum_{i,j=1}^n F_i(x) \frac{\partial}{\partial x_j} D_{ij}(\Hess v_1(x),\ldots,\Hess v_{n-1}(x)) \d x
\end{multline*}
and the statement follows from Lemma~\ref{le:sum_diff_c_ij_zero}.

As in the proof of Lemma~\ref{le:change_int_mixed_dis}, the general case follows from the fact that $F$ has compact support combined with a standard approximation argument and the dominated convergence theorem.
\end{proof}

\subsection{Applications to Mixed Monge--Amp\`ere Integrals}
In the following we consider special integrals of mixed discriminants where the support function of the unit ball $\Bn$ appears repeatedly.

First, we show that such integrals are well-defined. Recall that $\Hess h_{\Bn}(x)$ exists for every $x\neq 0$.
We remark that throughout this subsection, Lebesgue integrals with respect to the standard Lebesgue measure on $\R^n$ are considered. 
\goodbreak
\begin{lemma}
\label{le:int_zeta_dis_finite}
Let $1\leq k \leq n$. If $\zeta\in C_b((0,\infty))$ is such that $\lim_{r\to 0^+} r^{k-1}\zeta(r)$ exists and is finite, then the integral
$$
\int_{\R^n}\big\vert \zeta(|x|) \det(\Hess v_1(x),\ldots,\Hess v_k(x),\Hess h_{\Bn}(x)[n-k])\big\vert\d x
$$
is well-defined and finite for every $v_1,\ldots,v_k\in C^2(\R^n)$.
\end{lemma}

\begin{proof}
Fix $v_1,\ldots,v_k \in C^2(\R^n)$ and let $w\in C(\R^n\backslash\{0\})$ be defined by
$$w(x)= |x|^{n-k} \det(\Hess v_1(x),\ldots,\Hess v_k(x),\Hess h_{\Bn}(x)[n-k])$$
for $x\in \R^n\backslash\{0\}$. By the multilinearity of the mixed discriminant and \eqref{Jn} the function $w$ is bounded on $\Bn \backslash \{0\}$.
Using polar coordinates, we obtain
$$
\int\limits_{\R^n} \big\vert \zeta(|x|)  \det(\Hess v_1(x),\ldots,\Hess v_k(x),\Hess h_{\Bn}(x)[n-k])\big\vert\d x = \int\limits_{\sn} \int\limits_0^{\infty} \big\vert r^{k-1} \zeta(r) w(ry)\big\vert \d r \d\hm^{n-1}(y),
$$
where $\hm^{n-1}$ denotes the $(n-1)$-dimensional Hausdorff measure.
The result now follows from our assumptions on $\zeta$ together with the fact that $w$ is bounded on $\Bn\backslash\{0\}$.
\end{proof}

\goodbreak
The following result shows how replacing $h_{\Bn}$ by $\tfrac12 h_{\Bn}^2$ once in the mixed discriminant corresponds to taking an integral transform of the density function.

\begin{lemma}
\label{le:int_mixed_dis_r_transf}
Let $1\leq k \leq n-1$ and let $\varepsilon>0$. If $v_1,\ldots,v_k\in C^2(\R^n)$ and $\Hess v_1(x)=0$ for every $x\in \varepsilon \Bn$, then
\begin{multline*}
\int_{\R^n} \psi(|x|)  \det(\Hess v_1(x),\ldots,\Hess v_k(x),\Hess h_{\Bn}(x)[n-k])\d x\\
= \int_{\R^n}\rho(|x|)  \det(\Hess v_1(x),\ldots,\Hess v_k(x), \Hess h_{\Bn}(x)[n-k-1], I_n) \d x
\end{multline*}
for every $\psi\in C_b^2((0,\infty))$, where $\rho\in C_b((0,\infty))$ is given for $s>0$ by
$$\rho(s):= \frac{\psi(s)}{s}-\int_s^{\infty} \frac{\psi(t)}{t^2}\d t.$$

\end{lemma}

\goodbreak
\begin{proof}
Observe that our assumptions on $v_1$ imply that the mixed discriminants in both integrals vanish on $\varepsilon \Bn$. Since the support of $\psi$ is bounded, this implies that both integrals are well-defined and finite.

Let $\xi(t)=\int_t^{\infty} \frac{\psi(s)}{s^2} \d s$ for $t>0$. Since $\psi\in C_b^2((0,\infty))$ we have $\xi\in C_b^3((0,\infty))$ and furthermore $\psi(t)=-\xi'(t) t^2$ as well as $\frac{\psi(t)}{t}=-\xi'(t)t$ for $t>0$. Thus, we need to show that
\begin{align}
    \begin{split}
        \label{eq:int_mixed_dis_r_transf_subst}
        \int_{\R^n} \xi'(|x|)&|x|^2 \det(\Hess v_1(x),\ldots,\Hess v_{k}(x),\Hess h_{\Bn}(x)[n-k]) \d x \\[-4pt]
        &= \int_{\R^n} \left(\xi'(|x|)|x|+\xi(|x|) \right)\det(\Hess v_1(x),\ldots,\Hess v_{k}(x),\Hess h_{\Bn}(x)[n-k-1],I_n)\d x
    \end{split}
\end{align}
for every $v_1,\ldots,v_{k}\in C^2(\R^n)$. Since the mixed discriminants in both integrals vanish on $\varepsilon \Bn$, we can replace $h_{\Bn}$ as well as $x\mapsto \xi'(|x|)|x|^2$ and $x\mapsto \xi'(|x|)|x|+\xi(|x|)$ by suitable functions in $C^2(\R^n)$ without changing the values of the integrals. Thus after applying Lemma~\ref{le:change_int_mixed_dis} and changing back to the original functions, we obtain that \eqref{eq:int_mixed_dis_r_transf_subst} is equivalent to
\begin{multline*}
\int_{\R^n} |x|\det(\Hess v_1(x),\ldots,\Hess v_{k}(x), \Hess h_{\Bn}(x)[n-k-1],\Hess(\xi'(|x|)|x|^2))\d x\\
= \int_{\R^n} \frac{|x|^2}{2}  \det(\Hess v_1(x),\ldots,\Hess v_{k}(x), \Hess h_{\Bn}(x)[n-k-1],\Hess (\xi'(|x|)|x|+\xi(|x|)))\d x.
\end{multline*}
Using the multilinearity of mixed discriminants, it suffices to show that
$$
\int_{\R^n} \det\big(\Hess v_1(x),\ldots,\Hess v_{k}(x),\Hess h_{\Bn}(x)[n-k-1], |x|\,\Hess(\xi'(|x|)|x|^2)- \frac{|x|^2}{2}\,\Hess (\xi'(|x|)|x|+\xi(|x|))\big)\d x
$$
vanishes.
Since we have
$$\Hess(\xi'(|x|)|x|^2) = \xi'''(|x|)\, x\otimes x + \xi''(|x|)|x|\, I_n + 3 \frac{\xi''(|x|)}{|x|}\, x\otimes x + 2\xi'(|x|)\, I_n$$
and
$$\Hess(\xi'(|x|)|x|+\xi(|x|)) = \frac{\xi'''(|x|)}{|x|} \, x\otimes x + \xi''(|x|)\, I_n + 2\frac{\xi''(|x|)}{|x|^2}\, x\otimes x +2\frac{\xi'(|x|)}{|x|}\, I_n-2\frac{\xi'(|x|)}{|x|^3} \, x\otimes x,$$
we obtain
\begin{align*}
|x|\,\Hess(\xi'(|x|)|x|^2)- \frac{|x|^2}{2}\,\Hess (\xi'(|x|)|x|+\xi(|x|)) &= -\frac 12 \Big(\psi'(|x|)I_n+ \frac{\psi''(|x|)}{|x|}\, x\otimes x \Big)\\
&= -\frac 12 \D(\psi'(|x|)x)
\end{align*}
for every $x\in\R^n\backslash\{0\}$, where $\D(\psi'(|x|)x)$ denotes the Jacobian of the vector field $x\mapsto \psi'(|x|)x$. The result now follows from Lemma~\ref{le:int_mixed_dis_zero} and the definition of $D_{ij}$, where we have used again that we may replace the integrands in a neighborhood of the origin.
\end{proof}

\goodbreak
In the next two statements we remove the regularity assumptions of the last result and treat the case where the support function of the unit ball $\Bn$ is replaced multiple times.

\begin{proposition}
\label{prop:mixed_int_r_transf}
If $\,1\leq k \leq n-1$ and $\psi\in \Had{n-k+1}{n}$, then
\begin{multline*}
\int_{\R^n} \psi(|x|)  \det(\Hess v_1(x),\ldots,\Hess v_k(x),\Hess h_{\Bn}(x)[n-k])\d x\\[-2pt]
= \int_{\R^n} \cR^{-1}\! \psi(|x|)  \det(\Hess v_1(x),\ldots,\Hess v_k(x), \Hess h_{\Bn}(x)[n-k-1], I_n) \d x
\end{multline*}
for every $v_1,\ldots,v_k\in C^2(\R^n)$, where $\cR^{-1}$ was defined in Subsection \ref{ss:integral transform}.
\end{proposition}

\goodbreak

\begin{proof}
Since $\psi\in\Had{n-k+1}{n}$, there exists $\gamma>0$ such that $\psi(t)=0$ for every $t\geq \gamma$. Note that by Lemma~\ref{le:r_kln} this implies that $\cR^{-1} \!\psi(t)=0$ for every $t\geq \gamma$. We will assume first that there exists $\varepsilon>0$ such that $\Hess v_1(x)=0$ for every $x\in\varepsilon \Bn$.

Let $\psi_\varepsilon\in C_b((0,\infty))$ be such that $\psi_\varepsilon\equiv \psi$ on $[\varepsilon,\infty)$ and $\psi_\varepsilon\equiv 0$ on $(0,\varepsilon/2]$. 
Observe that this implies that $\cR^{-1}\!\psi_\varepsilon\equiv \cR^{-1}\!\psi$ on $[\varepsilon,\infty)$. For $\delta>0$ we can find $\psi_{\varepsilon,\delta}\in C_b^2((0,\infty))$ such that $\psi_{\varepsilon,\delta}\equiv 0$ on $(0,\varepsilon/2]\cup[\gamma+\delta,\infty)$ and such that $\psi_{\varepsilon,\delta}\to\psi_\varepsilon$ uniformly on $(\varepsilon/2,\gamma+\delta)$ (and thus on $(0,\infty)$) as $\delta\to 0^+$. 
By the properties of $\psi_\varepsilon$ this also implies uniform convergence of $\cR^{-1}\!\psi_{\varepsilon,\delta}$ to $\cR^{-1}\!\psi_\varepsilon$ on $(0,\infty)$ as $\delta\to 0^+$. 
Using that $\Hess v_1\equiv 0$ on $\varepsilon \Bn$, Lemma~\ref{le:int_mixed_dis_r_transf}, as well as the fact that the integrands in each of the following integrals are continuous and have compact support, we now have
\begin{align*}
\int_{\R^n} &\psi(|x|)\det(\Hess v_1(x),\ldots,\Hess v_k(x),\Hess h_{\Bn}(x)[n-k])\d x\\
&=\int_{\R^n} \psi_\varepsilon(|x|) \det(\Hess v_1(x),\ldots,\Hess v_k(x),\Hess h_{\Bn}(x)[n-k])\d x\\
&=\lim_{\delta\to 0^+}\int_{\R^n} \psi_{\varepsilon,\delta}(|x|)  \det(\Hess v_1(x),\ldots,\Hess v_k(x),\Hess h_{\Bn}(x)[n-k])\d x\\
&=\lim_{\delta\to 0^+} \int_{\R^n} \cR^{-1}\! \psi_{\varepsilon,\delta}(|x|)  \det(\Hess v_1(x),\ldots,\Hess v_k(x), \Hess h_{\Bn}(x)[n-k-1], I_n) \d x\\
&= \int_{\R^n} \cR^{-1}\! \psi_\varepsilon(|x|)  \det(\Hess v_1(x),\ldots,\Hess v_k(x), \Hess h_{\Bn}(x)[n-k-1], I_n) \d x\\
&=\int_{\R^n} \cR^{-1}\! \psi(|x|)  \det(\Hess v_1(x),\ldots,\Hess v_k(x), \Hess h_{\Bn}(x)[n-k-1], I_n) \d x,
\end{align*}
which completes the proof under the additional assumptions on $v_1$.

For general $v_1\in C^2(\R^n)$, observe that without loss of generality we may assume that $v_1(0)=0$ and $\nabla v_1(0)=0$. Thus, there exist $\beta,\varepsilon_0>0$ such that $|v_1(x)|\leq \beta |x|^2$ and $|\nabla v_1(x)|\leq \beta |x|$ for every $x\in 2\varepsilon_0 \Bn$. Let $\varphi\in C^2([0,\infty))$ be such that $\varphi(t)=0$ for $t\in[0,1]$ and $\varphi(t)=1$ for $t\in[2,\infty)$. For $\varepsilon\in(0,\varepsilon_0)$, set $v_{1,\varepsilon}(x)\!:=v_1(x)\, \varphi(|x|/\varepsilon)$ for $x\in\R^n$. We now have $\Hess v_{1,\varepsilon}(x)= 0$ for every $x\in \varepsilon \Bn$ and our assumptions on $v_1$ together with the fact that $\varphi$ is constant on $[2,\infty)$ imply that $\Hess v_{1,\varepsilon}$ is uniformly bounded on $\gamma \Bn$ for every $\varepsilon\in (0,\varepsilon_0)$. Moreover,  $\Hess v_{1,\varepsilon}\to \Hess v_1$ pointwise on $\R^n$ as $\varepsilon\to 0^+$. Since $\psi\in\Had{n-k+1}{n}$ the limit $\lim_{t\to 0^+} t^{k-1}\psi(t)$ exists and is finite. By Lemma~\ref{le:r_kln} and since $\Had{n-k+1}{n}=\Had{n-k}{n-1}$, we have $\cR^{-1}\!\psi \in \Had{n-k}{n}$ and thus also $\lim_{t\to 0^+} t^{k}\cR^{-1}\!\psi(t)$ exists and is finite.
Hence, by the first part of the proof and Lemma~\ref{le:int_zeta_dis_finite} combined with the dominated convergence theorem we now obtain
\begin{align*}
\int_{\R^n} &\psi(|x|)\det(\Hess v_1(x),\ldots,\Hess v_k(x),\Hess h_{\Bn}(x)[n-k])\d x\\
&=\lim_{\varepsilon\to 0^+} \int_{\R^n} \psi(|x|)\det(\Hess v_{1,\varepsilon}(x),\Hess v_2(x),\ldots,\Hess v_k(x),\Hess h_{\Bn}(x)[n-k])\d x\\
&=\lim_{\varepsilon\to 0^+} \int_{\R^n} \cR^{-1} \!\psi(|x|)  \det(\Hess v_{1,\varepsilon}(x),\Hess v_2(x),\ldots,\Hess v_k(x), \Hess h_{\Bn}(x)[n-k-1], I_n) \d x\\
&=\int_{\R^n} \cR^{-1}\! \psi(|x|)  \det(\Hess v_1(x),\Hess v_2(x),\ldots,\Hess v_k(x), \Hess h_{\Bn}(x)[n-k-1], I_n) \d x,
\end{align*}
which concludes the proof.
\end{proof}

\goodbreak
\begin{proposition}
\label{prop:transform_ints}
If $\,1\leq j \leq n-1$ and $\zeta\in\Had{j}{n}$, then 
$$
\int_{\R^n}\zeta(\vert x\vert) \,[\Hess v(x)]_{j}\d x=
\binom{n}{j} \int_{\R^n}\cR^{n-j}\! \zeta (\vert x\vert) \det(\Hess v(x)[j],\Hess h_{\Bn}(x)[n-j])\d x
$$
for every $v\in C^2(\R^n)$.
\end{proposition}

\goodbreak
\begin{proof}
Let $1\leq j \leq n-1$ and $\zeta\in\Had{j}{n}$ be given. We claim that
\begin{align}
\begin{split}
\label{eq:claim_r_n-k}
\int_{\R^n}\cR^{n-k}& \zeta(\vert x\vert) \det(\Hess v(x)[j],\Hess h_{\Bn}(x)[n-k], I_n[k-j])\d x\\
&= \int_{\R^n} \cR^{n-(k+1)} \zeta(\vert x\vert) \det(\Hess v(x)[j], \Hess h_{\Bn}(x)[n-(k+1)], I_n[(k+1)-j])\d x
\end{split}
\end{align}
for every $k\in\N$ such that $j\leq k \leq n-1$ and every $v\in C^2(\R^n)$. Indeed, as $\zeta\in\Had{j}{n}$ it follows from Lemma~\ref{le:r_kln} that $\cR^{n-k}\zeta\in\Had{j}{k}$. Since $\Had{j}{k}=\Had{n-k+j}{n}$ and $\Had{n-k+j}{n}\subseteq \Had{n-k+1}{n}$, the claim now follows from Proposition~\ref{prop:mixed_int_r_transf}.

Applying \eqref{eq:claim_r_n-k} recursively $(n-j)$ times (for each possible value of $k$), we obtain that
$$
\int_{\R^n}\cR^{n-j} \zeta(\vert x \vert) \det(\Hess v(x)[j], \Hess h_{\Bn}(x)[n-j])\d x\\
=\int_{\R^n} \zeta(\vert x \vert) \det(\Hess v(x)[j], I_n[n-j])\d x
$$
for every $v\in C^2(\R^n)$. The statement now follows from \eqref{eq:mixed_dis_hessian}.
\end{proof}
\goodbreak

\section{Proofs of the Main Results}
\label{se:proofs}
In this section, we present a new proof of the existence of functional intrinsic volumes, Theorem~\ref{thm:existence_singular_hessian_vals} and Theorem~\ref{thm:existence_singular_hessian_vals_dual}. Moreover, we prove our main results: the new representations of functional intrinsic volumes, Theorem~\ref{thm:steiner_measures} and Theorem~\ref{thm:steiner_measures_dual}, as well as the Steiner formulas, Theorem~\ref{thm_steiner_functions} and Theorem~\ref{thm_steiner_functions_dual}.

By the duality relations between valuations on $\fconvs$ and $\fconvf$, it is enough to  prove the results for valuations on $\fconvf$, that is, to prove Theorem~\ref{thm:existence_singular_hessian_vals_dual}, Theorem~\ref{thm_steiner_functions_dual} and Theorem~\ref{thm:steiner_measures_dual}. Theorem~\ref{thm:existence_singular_hessian_vals} and Theorem~\ref{thm_steiner_functions} are immediate consequences of their counterparts on $\fconvf$ while Theorem~\ref{thm:steiner_measures} follows from Theorem~\ref{thm:steiner_measures_dual} combined with Theorem~\ref{representation of Steiner measures primal smooth}.

\subsection{New Proof of Theorem~\ref{thm:existence_singular_hessian_vals_dual}}
The statement is trivial for $j=0$ and follows from 
Proposition~\ref{prop:mixed_hessian_vals}
for $j=n$. So, let $1\leq j \leq n-1$ and  $\zeta\in\Had{j}{n}$. We set $\alpha\!:=\binom{n}{j}\cR^{n-j} \zeta$ and note that, by Lemma \ref{le:r_kln}, we have $\alpha\in \Had{n}{n}$. We define  $\oZ\colon\fconvf\to\R$  by
$$\oZ(v):=\int_{\R^n} \alpha(|x|) \d\nhd{j}(v;x).$$
By Proposition~\ref{prop:mixed_hessian_vals}, the definition of $\nhd{j}(v;\cdot)$, and \eqref{def nhd}, the functional $\oZ$ is a continuous and dually epi-translation invariant valuation on $\fconvf$. It is easy to see that $\oZ$ is rotation invariant.
For $v\in \fconvf\cap C^2(\R^n)$, it follows from Theorem~\ref{properties of nhm} \ref{nhm 0} and \ref{nhm c} that
$$\oZ(v)= \int_{\R^n} \alpha(|x|)  \det(\Hess v(x)[j],\Hess h_{\Bn}(x)[n-j])\d x$$
and thus, by Proposition~\ref{prop:transform_ints}, the valuation $\oZ$ satisfies \eqref{eq:rep_ozz_c2_dual} for $v\in \fconvf\cap C^2(\R^n)$.

We conclude that  $\oZ$ has the required properties and remark that it is uniquely determined by \eqref{eq:rep_ozz_c2_dual}, since $\fconvf\cap C_+^2(\R^n)$ is dense in $\fconvf$.

\subsection{Proof of Theorem~\ref{thm:steiner_measures_dual}}
For $j=n$ the statement trivially follows from 
Proposition~\ref{prop:mixed_hessian_vals}
and Theorem~\ref{properties of nhm} \ref{nhm c}. Next, consider the case $j=0$ and let $\zeta\in\Had{0}{n}$ and $\alpha\in C_c([0,\infty))$ be as in the statement of the theorem. Using polar coordinates and \eqref{MA of the norm} we now have
\begin{multline*}
\oZZd{0}{\zeta}(v)=\int_{\R^n}\zeta(|x|)\d x = n \kappa_n \lim_{s\to 0^+} \int_s^{\infty} t^{n-1} \zeta(t) \d t\\
= \kappa_n\,\alpha(0)= \int_{\R^n} \alpha(|x|) \d\MA(h_{\Bn}; x)= \int_{\R^n} \alpha(|x|)\d\nhd{0}(v;x)
\end{multline*}
for every $v\in\fconvf$, which proves the statement for $j=0$. Finally, let $1\leq j \leq n-1$ and $\zeta\in\Had{j}{n}$. For $v\in\fconvf\cap C^2(\R^n)$, it follows from \eqref{eq:rep_ozz_c2_dual}, which was established in  the new proof of Theorem~\ref{thm:existence_singular_hessian_vals_dual}  for $v\in\fconvf\cap C^2(\R^n)$, and Proposition~\ref{prop:transform_ints} that
\begin{align*}
\oZZd{j}{\zeta}(v)&=\int_{\R^n} \zeta(|x|)\big[\Hess v(x)\big]_{j} \d x\\
&= \binom{n}{j} \int_{\R^n}\cR^{n-j} \zeta (\vert x\vert) \det(\Hess v(x)[j],\Hess h_{\Bn}(x)[n-j])\d x.
\end{align*}
Combining this with Theorem~\ref{properties of nhm} \ref{nhm 0} and \ref{nhm c}, we obtain
$$\oZZd{j}{\zeta}(v) = \int_{\R^n} \alpha(|x|) \d\nhd{j}(v;x).$$
The statement now follows from Proposition~\ref{prop:mixed_hessian_vals} and the fact that $\fconvf\cap C^2(\R^n)$ is dense in $\fconvf$.

\subsection{Proof of Theorem~\ref{thm_steiner_functions_dual}}
Let $\zeta\in\Had{n}{n}$ be given and for $0\leq j\leq n$, let $\zeta_j\in\Had{j}{n}$ be defined as in \eqref{eq:steiner_function_dual_transform}. By Theorem~\ref{thm:steiner_measures_dual} and \eqref{eq:steiner_formula_maj} we have
\begin{align*}
\oZZ{n}{\zeta}^*(v+r h_{\Bn}) &= \int_{\R^n} \zeta(|x|) \d \MA(v+r h_{\Bn};x)\\
&= \sum_{j=0}^n \binom{n}{j} r^{n-j} \int_{\R^n} \zeta(|x|) \d \nhd{j}(v;x)
\end{align*}
for every $v\in\fconvf$ and $r>0$. Using Theorem~\ref{thm:steiner_measures_dual} again and Lemma~\ref{le:r_kln}, we obtain that
$$
\binom{n}{j} \int_{\R^n} \zeta(|x|) \d \nhd{j}(v;x)= \kappa_{n-j} \oZZ{j}{\frac{1}{\kappa_{n-j}}\cR^{-(n-j)} \zeta}^*(v) = \kappa_{n-j} \oZZ{j}{\zeta_j}^*(v)
$$
for every $0\leq j \leq n$ and $v\in\fconvf$, which concludes the proof.

\section{Additional Results and Applications}\label{additional results}

In this section we prove additional results and derive further applications. Subsection~\ref{ss:alternate_proof_of_steiner} contains a second proof of the functional Steiner formula, Theorem~\ref{thm_steiner_functions}, which uses the Hadwiger Theorem on convex functions, Theorem~\ref{thm:hadwiger_convex_functions}. In the subsequent subsection we use the properties of mixed Monge--Amp\`ere measures to obtain a new representation of functional intrinsic volumes. In Subsection~\ref{ss:retrieve_steiner}, we show how the classical Steiner formula \eqref{steiner} can be retrieved from our new functional version. The final subsection partly answers the question, which functions playing the role of the unit ball give rise to all functional intrinsic volumes in a Steiner-type formula.

\subsection{Alternate Proof of Theorem~\ref{thm_steiner_functions}}
\label{ss:alternate_proof_of_steiner}
Our approach for this proof follows the proof of the classical Steiner formula from \cite[Theorem 9.2.3]{Klain:Rota} and uses Theorem~\ref{thm:hadwiger_convex_functions}. 
We remark that multinomiality with respect to infimal convolution of continuous, epi-translation invariant valuations on $\fconvs$ was established by the authors in \cite{Colesanti-Ludwig-Mussnig-4}. For polynomial expansions for a different functional analog of the volume on convex functions, see Milman and Rotem \cite{milman_rotem}.

Let $\zeta\in\Had{n}{n}$ be given. It is easy to see that $u\mapsto \oZZ{n}{\zeta}(u\infconv \ind_{\Bn})$ defines a continuous, epi-translation and rotation invariant valuation on $\fconvs$. Thus, by Theorem~\ref{thm:hadwiger_convex_functions} there exist functions $\tilde{\zeta}_j\in\Had{j}{n}$ for $0\leq j \leq n$ such that
$$\oZZ{n}{\zeta}(u\infconv \ind_{\Bn}) = \sum_{j=0}^n \oZZ{j}{\tilde{\zeta}_j}(u)$$
for every $u\in\fconvs$. Using the epi-homogeneity of functional intrinsic volumes, we now have
$$\oZZ{n}{\zeta}(u\infconv (r\sq \ind_{\Bn})) = r^n \oZZ{n}{\zeta}((\tfrac 1r \sq u) \infconv \ind_{\Bn}) = r^n \sum_{j=0}^n \oZZ{j}{\tilde{\zeta}_j}(\tfrac 1r \sq u) = \sum_{j=0}^n r^{n-j} \oZZ{j}{\tilde{\zeta}_j}(u)$$
for every $u\in\fconvs$ and $r>0$.

In order to determine the functions $\tilde{\zeta}_j$ for $0\leq j \leq n$, we  evaluate the last expression for $u=u_t$ with $t> 0$, where $u_t(x)\!:= t \vert x\vert + \ind_{\Bn}(x)$ for $x\in\R^n$. Since
$$\big(u_t\infconv(r \sq \ind_{\Bn})\big)(x) = \begin{cases}
 0\quad &\text{for } 0\leq |x|\leq r,\\
 t(|x|-r)\quad &\text{for } r < |x| \leq 1+r,\\
+\infty \quad &\text{for } 1+r < |x|,
 \end{cases}$$
a simple calculation shows that
\begin{align*}
\oZZ{n}{\zeta}(u_t\infconv(r \sq \ind_{\Bn})) &= \int_{(1+r) \Bn} \zeta(|\nabla (u_t \infconv (r\sq \ind_{\Bn}))(x)|) \d x\\
&= \kappa_n r^n \zeta(0) + n\kappa_n \zeta(t) \int_{r}^{1+r} s^{n-1} \d s\\
&= \kappa_n r^n \zeta(0) + \sum_{j=1}^{n} \binom{n}{j} r^{n-j} \kappa_n \zeta(t)
\end{align*}
for every $r>0$ and $t> 0$. A comparison of coefficients combined with Lemma~\ref{le:calc_ind_bn_tx_theta_i} shows that $\cR^{n-j}\tilde{\zeta}_j(t)=\zeta(t)$ for every $t> 0$ and $1\leq j \leq n$. Thus, by Lemma~\ref{le:r_kln}, we get $\tilde{\zeta}_j=\cR^{-(n-j)} \zeta$ for every $1\leq j \leq n$.

For $j=0$, observe that
$$\oZZ{0}{\xi}(u)=\oZZd{0}{\xi}(u^*)=\int_{\R^n} \xi(|x|) \d x = n \, \kappa_n \lim_{s\to 0^+} \int_s^{\infty} t^{n-1} \xi(t) \d t$$
for every $u\in\fconvs$ and $\xi\in\Had{0}{n}$. Thus, our calculations combined with Lemma~\ref{le:r_kln} and the definition of $\Had{0}{n}$ show that
$$\kappa_n \zeta(0)=\oZZ{0}{\tilde{\zeta}_0}(u_t)=n \kappa_n \int_0^{\infty} \tilde{\zeta}_0(s) s^{n-1} \d s = \kappa_n \cR^{n} \tilde{\zeta}_0(0)$$
for every $t>0$. Since $\oZZ{0}{\tilde{\zeta}_0}(u)$ is independent of $u\in\fconvs$ and only depends on $\cR^n \tilde{\zeta}_0(0)$, it easily follows from Lemma~\ref{le:r_kln} that we may choose $\tilde{\zeta}_0 = \cR^{-n} \zeta$.

The result now follows by setting $\zeta_j=\frac{1}{\kappa_{n-j}} \tilde{\zeta}_j = \frac{1}{\kappa_{n-j}} \cR^{-(n-j)} \zeta$ and observing that
$$\oZZ{j}{\tilde{\zeta}_j}(u) = \kappa_{n-j} \oZZ{j}{\zeta_j}(u)$$
for every $u\in\fconvs$ and $0\leq j \leq n$.

\goodbreak
\subsection{Representation Formulas for Functional Intrinsic Volumes}

Let  $1\le j \le n$.
By Theorem \ref{thm:steiner_measures_dual}, 
\begin{equation*}
\oZZd{j}{\zeta}(v)=  \int_{\R^n} \alpha(|x|)\d\nhd{j}(v;x)
\end{equation*}
for $v\in \fconvf$ and $\zeta\in\Had{j}{n}$, where $\alpha\!: = \binom{n}{j} \cR^{n-j}\!\zeta$.  If, in addition, $v\in C^2(\R^n\backslash \{0\})$, then Theorem \ref{properties of nhm} \ref{nhm 0} and \ref{nhm c}  imply that
\begin{equation}\label{smooth_dual}
\oZZd{j}{\zeta}(v)= \alpha(0)\, V_{j}(\partial v(0)) +\int_{\R^n} \alpha(|x|)\,\det(\Hess v(x)[j],\Hess h_{\Bn}(x)[n-j])\d x.
\end{equation}
Correspondingly, by Theorem \ref{thm:steiner_measures}, we obtain  that
\begin{equation*}
\oZZ{j}{\zeta}(u)= \int_{\R^n}\alpha(|y|) \d\nhp{j}(u;y)
\end{equation*}
for $u\in\fconvs$ and $\zeta\in\Had{j}{n}$, where $\alpha$ is defined as before.
If, in addition, $u\in C_+^2(\R^n\backslash \argmin u)$, then Theorem~\ref{properties of nhp} \ref{nhm 0} and \ref{nhm c} combined with Theorem \ref{representation of Steiner measures primal smooth} imply that
\begin{equation}\label{smooth_primal}
\oZZ{j}{\zeta}(u)= \alpha(0)\, V_{j}(\argmin u) +\frac1{\binom{n}{j}}\int_{\R^n} \alpha(|\nabla u(x)|)\,\tau_{n-j}(u,x)\d x.
\end{equation}
Note that $\alpha$ can be extended to a function in $C_c([0,\infty))$ which implies that the densities in the integrals in \eqref{smooth_dual} and \eqref{smooth_primal} are not singular. Hence, in the special cases considered here, we obtain a representation of functional intrinsic volumes as Hessian valuations with continuous densities and an additional term involving classical intrinsic volumes.

\subsection{Retrieving the Classical Steiner Formula}
\label{ss:retrieve_steiner}
As a further application of Theorem~\ref{thm_steiner_functions}, we retrieve the classical Steiner formula \eqref{steiner} from \eqref{steiner_function}.
We need the following result, which shows how the classical intrinsic volumes can be retrieved from the functional intrinsic volumes.

\begin{proposition}[\!\cite{Colesanti-Ludwig-Mussnig-5}, Proposition 5.2]
\label{prop:retrieve_intrinsic_volumes}
If $\,0\leq j \leq n-1$ and $\zeta\in\Had{j}{n}$, then
$$\oZZ{j}{\zeta}(\ind_K) = \kappa_{n-j}  \cR^{n-j}\!\zeta(0)\,V_j(K)$$
for every $K\in\cK^n$.
If $\zeta\in\Had{n}{n}$, then 
$$\oZZ{n}{\zeta}(\ind_K) = \zeta(0)\, V_n(K)$$ 
for every $K\in\cK^n$.
\end{proposition}

Let $r>0$ and choose $u$ to be the convex indicator function of a convex body $K\in\cK^n$. We have
$$u\infconv(r\sq \ind_{\Bn}) = \ind_K \infconv(r \sq \ind_{\Bn}) = \ind_{K+r \Bn}$$
and therefore Theorem~\ref{thm_steiner_functions} combined with  Lemma~\ref{le:r_kln} and Proposition~\ref{prop:retrieve_intrinsic_volumes} implies that
$$\zeta(0)V_n(K+r \Bn) = \oZZ{n}{\zeta}(\ind_K\infconv (r\sq \ind_{\Bn})) = \sum_{j=0}^n r^{n-j} \oZZ{j}{\cR^{-(n-j)}\zeta}(\ind_K)= \zeta(0) \sum_{j=0}^n r^{n-j}\kappa_{n-j} V_j(K)$$
for every $K\in\cK^n$ and $\zeta\in\Had{n}{n}$, which gives the classical Steiner formula if $\zeta(0)\ne0$.

\subsection{General Functional Steiner Formulas}
\label{ss:further_steiner_formulas}
We remark that the proof of Theorem~\ref{thm_steiner_functions} shows that Steiner formulas for convex functions are also obtained if we replace the convex indicator function $\ind_{\Bn}$ by any  radially symmetric, super-coercive, convex function. Similarly, the support function $h_{\Bn}$ in Theorem~\ref{thm_steiner_functions_dual} can be replaced by any radially symmetric, finite-valued, convex function. However, in general such formulas do not give rise to all functional intrinsic volumes $\oZZd{j}{\zeta}$, that is, not all $\zeta\in\Had{j}{n}$ will appear in the polynomial expansion. For example, if $v\in\fconvf\cap C_+^2(\R^n)$, then it easily follows from \eqref{eq:rep_ozz_c2_dual} and the definition of  mixed discriminant that 
\begin{equation}\label{dualsteiner}
\oZZd{n}{\zeta}(v+\tfrac12\,r\, h_{\Bn}^2) = \sum_{j=0}^n r^{n-j} \oZZd{j}{\zeta}(v)
\end{equation}
for every $\zeta\in\Had{n}{n}$ and $r>0$. By continuity, \eqref{dualsteiner} also holds for all $v\in\fconvf$. Here we use that $\Had{n}{n}\subseteq \Had{j}{n}$ for every $0\leq j \leq n$ to show that the functional intrinsic volumes appearing in \eqref{dualsteiner} are well-defined. However, the classes $\Had{n}{n}$ and $\Had{j}{n}$ do not coincide if $j<n$, which shows that not all functional intrinsic volumes $\oZZd{n}{\zeta_j}$ with $\zeta_j\in\Had{j}{n}$ are obtained in this way.

This raises the question for which convex functions $\phi\colon [0,\infty)\to\R$ we obtain all functional intrinsic volumes when we replace $h_{\Bn}$ by $\phi\circ h_{\Bn}$. 
Let $\Vconvf{j}$ be the set of continuous, dually epi-translation and rotation invariant valuations on $\fconvf$ that are homogeneous of degree $j$. By  Theorem \ref{thm:hadwiger_convex_functions}, we know that 
$$\Vconvf{j}=\{\oZZd{j}{\zeta}\colon \zeta\in\Had{j}{n}\}$$
for $0\le j\le n$.
We obtain the following complete description if we use a regularity assumption for $\phi$.

\begin{theorem}\label{general unit ball}  Let $\phi\in C^2([0,\infty))$ be convex and such that $\phi'(0)\ge0$. For $1\le j\le n-1$, 
$$
\Vconvf{j}=\Big\{\int\nolimits_{\R^n}\beta(|x|)\d\MA(v [j], \phi\circ h_{\Bn}[n-j];x)\colon \beta\in C_c([0,\infty))\Big\},
$$
if and only if
$\phi'(0)>0$.
\end{theorem}

We require the following results for the proof of Theorem \ref{general unit ball}. The function $v_t$ is defined in \eqref{vt def}.
\begin{lemma}
\label{le:v_t_determines}
Let $\oZ_1,\oZ_2\colon\fconvf\to\R$ be continuous, dually epi-translation and rotation invariant valuations that are homogeneous of degree $j$ with $0\leq j \leq n$. If $\oZ_1(v_t)=\oZ_2(v_t)$ for every $t\geq 0$, then $\oZ_1\equiv \oZ_2$.
\end{lemma}
\begin{proof}
By Theorem~\ref{thm:hadwiger_convex_functions_dual}, there exist $\zeta_{1}, \zeta_{2}\in\Had{j}{n}$ such that
$$\oZ_1(v)=\oZZd{j}{\zeta_{1}}(v) \qquad \text{and}\qquad \oZ_2(v)=\oZZd{j}{\zeta_{2}}(v)$$ for every $v\in\fconvf$.
If $j=0$, then both $\oZ_1$ and $\oZ_2$ are constants, independent of $v$, and thus the statement is trivial. For $1\leq j\leq n$, it follows from Lemma~\ref{le:ozzd_v_t} and our assumptions on $\oZ_1$ and $\oZ_2$ that
$$\kappa_n \binom{n}{j} \cR^{n-j} \zeta_1(t)=\oZZd{j}{\zeta_1}(v_t)=\oZ_1(v_t)=\oZ_2(v_t) = \oZZd{j}{\zeta_2}(v_t)=\kappa_n \binom{n}{j} \cR^{n-j} \zeta_2(t)$$
for every $t\geq 0$. By Lemma~\ref{le:r_kln} this implies $\zeta_1\equiv \zeta_2$ and thus $\oZ_1\equiv \oZ_2$.
\end{proof}

\noindent
We remark that it would be of great interest to find a proof of the previous lemma that does not require Theorem~\ref{thm:hadwiger_convex_functions_dual}. In particular, this would provide a new strategy to prove the Hadwiger theorem for convex functions.

\goodbreak
\begin{lemma}\label{int_eq}
Let $1\le j\le n-1$, let $\phi\in C^2([0,\infty))$ be convex with $\phi'(0)\geq 0$ and $\beta\in C_c([0,\infty))$. If the functional $\,\bar\oZ\colon\fconvf\to \R$ is given by
\begin{equation}
\label{eq:z_equals_z_phi_beta}
\bar\oZ(v):= \int_{\R^n} \beta(|x|) \d\MA(v[j], \phi \circ h_{\Bn} [n-j]; x),
\end{equation}
then
\begin{equation}\label{integral equation}
\kappa_n\Big(\beta(t)\phi'(t)^{n-j}+(n-j)\int_t^{\infty}\beta(r)\phi'(r)^{n-j-1}\phi''(r)\d r\Big)=\bar\oZ(v_t)
\end{equation}
for $t\ge0$.
\end{lemma}

\begin{proof}
First, let $\phi,\psi\in C^2([0,\infty))$ be convex and such that $\psi'(0)=0$. 
We want to compute the mixed discriminant
$$
\det\big(\Hess (\phi\circ h_{\Bn})[n-j],\Hess (\psi\circ h_{\Bn})[j]\big).
$$
For $x\in\R^n$, set $r\!:=\vert x\vert$. For $r>0$, by the radial symmetry of $\phi\circ h_{\Bn}$ and $\psi\circ h_{\Bn}$ and by choosing a coordinate system such that $e_n$ is parallel to $x$, we obtain
$$
\Hess (\phi\circ h_{\Bn})(x)=\diag\Big(\frac{\phi'(r)}{r},\dots,\frac{\phi'(r)}{r},\phi''(r)\Big),\quad
\Hess (\psi\circ h_{\Bn})(x)=\diag\Big(\frac{\psi'(r)}{r},\dots,\frac{\psi'(r)}{r},\psi''(r)\Big),
$$
where $\diag(\lambda_1,\dots,\lambda_n)$ is the $n\times n$ diagonal matrix with entries $\lambda_1, \dots,\lambda_n$ in the diagonal.
Therefore, for $\varepsilon>0$,
$$
\det\big(\Hess (\phi\circ h_{\Bn})(x)+\varepsilon\,\Hess (\psi\circ h_{\Bn})(x)\big)=
\Big(\frac{\phi'(r)}{r}+\varepsilon\frac{\psi'(r)}{r}\Big)^{n-1}(\phi''(r)+\varepsilon\psi''(r)).
$$
Using the previous expression and \eqref{polarisation md}, we obtain, after some computations, that
$$
\det\big(\Hess (\phi\circ h_{\Bn})(x)[n-j],\Hess (\psi\circ h_{\Bn})(x)[j]\big)=\frac 1{n r^{n-1}}\left(\phi'(r)^{n-j}\psi'(r)^j\right)'.
$$ 

Next, assume that $\beta\in C^1_c([0,\infty))$.  By the previous step and Theorem~\ref{thm mixed MA measures} \ref{MMA c},
\begin{equation}\label{smooth case}
\begin{array}{rcl}
\displaystyle\int_{\R^n}\beta(|x|)\d\MA(\phi\circ h_{\Bn}[n-j],\psi\circ h_{\Bn}[j];x)&=&\displaystyle\kappa_n\int_0^\infty\beta(r)\left(\phi'(r)^{n-j}\psi'(r)^j\right)'\d r\\
&=&\displaystyle-\kappa_n\int_0^\infty\beta'(r)\phi'(r)^{n-j}\psi'(r)^j\d r
\end{array}
\end{equation}
where we used polar coordinates, integration by parts and the condition $\psi'(0)=0$. For $t>0$, set 
$$
\psi_t(r):=\max\{0,r-t\}
$$
for $r>0$. Note that for $v_t$, we have
$$v_t=\psi_t\circ h_{\Bn}.$$
For $t>0$, there exists a sequence of convex functions $\psi_{t,j}$ converging to $\psi_t$ uniformly in $[0,\infty)$ and such that $\psi_{t,j}\in C^2([0,\infty))$ and $\psi_{t,j}'(0)=0$ for every $j$. Moreover, the sequence $\psi_{t,j}$ can be chosen so that $\psi'_{t,j}$ is uniformly bounded and converges pointwise to $\psi_t'$ in $[0,\infty)$ except for $r=t$. By \eqref{smooth case}, the weak continuity of Monge--Amp\`ere measures, the fact that the support of $\beta$ is bounded, and the dominated convergence theorem, we obtain that
\begin{eqnarray*}
\int_{\R^n}\beta(|x|)\d\MA(v_t[j],\phi\circ h_{\Bn}[n-j];x)=-\kappa_n\int_t^{\infty}\beta'(r)\phi'(r)^{n-j}\d r.
\end{eqnarray*}
Integration by parts gives
\begin{equation*}
\int_{\R^n}\beta(|x|)\d\MA(v_t[j],\phi\circ h_{\Bn}[n-j];x)=\kappa_n\big(\beta(t)\phi'(t)^{n-j}+(n-j)\int_t^{\infty}\beta(r)\phi'(r)^{n-j-1}\phi''(r)\d r\big).
\end{equation*}
This equation, which has been proved in the case $\beta\in C^1_c([0,\infty))$, can now be extended to the case that $\beta\in C_c([0,\infty))$ by approximating $\beta$ uniformly on its support by a sequence of functions in $C^1_c([0,\infty))$. 
\end{proof}

\goodbreak
Theorem~\ref{general unit ball} follows from the next two propositions.

\begin{proposition}\label{general unit ball - sufficiency} Let $\phi\in C^2([0,\infty))$ be convex and let $\phi'(0)>0$. If $\,\oZ\colon\fconvf\to\R$ is a continuous,  dually epi-translation and rotation invariant valuation that is homogeneous of degree $j$ with $1\le j\le n-1$, then there exists $\beta\in C_c([0,\infty))$ such that
$$
\oZ(v)=\int_{\R^n}\beta(|x|)\d\MA(v[j],\phi\circ h_{\Bn}[n-j];x)
$$
for every $v\in\fconvf$.
\end{proposition}

\begin{proof}
Given $\alpha\in C_c([0,\infty))$, define the function $\beta\colon[0,\infty)\to\R$ as
\begin{equation}\label{beta}
\beta(t):=-\frac1{\kappa_n}\left(\frac{\alpha(t)}{\phi'(t)^{n-j}}+(n-j)\int_t^\infty\frac{\alpha(r)}{\phi'(r)^{n-j+1}}\phi''(r)\d r\right),
\end{equation}
where we use that $\phi'(t)>0$ for every $t\in[0,\infty)$. Also note that $\beta\in C_c([0,\infty))$. We claim that $\beta$ is a solution of the equation
\begin{equation}\label{alpha}
\kappa_n\Big(\beta(t)\phi'(t)^{n-j}+(n-j)\int_t^{\infty}\beta(r)\phi'(r)^{n-j-1}\phi''(r)\d r\Big)=\alpha(t).
\end{equation}
If we assume that $\alpha\in C_c^1([0,\infty))$, then also $\beta\in C_c^1([0,\infty))$ and \eqref{alpha} can be written in the form
\begin{equation*}\label{integral equation 2}
-\kappa_n\int_t^\infty\beta'(r)\phi'(r)^{n-j}\d r=\alpha(t).
\end{equation*}
Hence the claim is easily verified under the additional assumption on $\alpha$. The general case is obtained by approximation.

For $t\ge0$, define 
$$
\alpha(t):=\oZ(v_t).
$$
By Theorem~\ref{thm:hadwiger_convex_functions_dual}, Lemma~\ref{le:ozzd_v_t} and Lemma~\ref{le:r_kln}, we know that $\alpha\in C_c([0,\infty))$. For this function $\alpha$, define $\beta\!:[0,\infty)\to \R$ by \eqref{beta}. Define $\bar{\oZ}\colon \fconvf\to\R$  as
$$
\bar{\oZ}(v):=\int_{\R^n}\beta(|x|)\d\MA(v[j], \phi\circ h_{\Bn}[n-j];x).
$$
By Proposition~\ref{prop:mixed_hessian_vals}, the functional $\bar{\oZ}$
is a continuous, dually epi-translation and rotation invariant valuation on $\fconvf$ that is homogeneous of degree $j$. By Lemma \ref{int_eq} and \eqref{alpha},
$$
\oZ(v_t)=\bar{\oZ}(v_t)
$$
for every $t\ge0$. By Lemma~\ref{le:v_t_determines}, this implies that $\oZ\equiv \bar{\oZ}$.
\end{proof}

\begin{proposition}\label{general unit ball - necessity} Let $1\le j\le n-1$. If $\phi\in C^2([0,\infty))$ is convex and $\phi'(0)=0$, then there exists a continuous, dually epi-translation and rotation invariant valuation $\oZ\colon\fconvf\to \R$ that is homogeneous of degree $j$  such that 
\begin{equation*}
\oZ(v)= \int_{\R^n} \beta(|x|) \d\MA(v[j],\phi \circ h_{\Bn} [n-j]; x) \quad \text{ for all } v\in\fconvf
\end{equation*}
is not verified by any $\beta\in C_c([0,\infty))$.
\end{proposition}

\begin{proof} Let $\alpha\in C^2_c([0,\infty))$ be such that $\alpha'(0)>0$. By Lemma~\ref{le:ozzd_v_t} and Lemma~\ref{le:r_kln}, there exists a continuous, dually epi-translation and rotation invariant valuation $\oZ$ on $\fconvf$ that is homogeneous of degree $j$ such that 
$$
\oZ(v_t)=\alpha(t)
$$
for $t\ge 0$. Assume that there exists $\beta\in C_c([0,\infty))$ such that \eqref{eq:z_equals_z_phi_beta} is satisfied for this functional $\oZ$. 
By Lemma \ref{int_eq}, 
the function $\beta$ has to verify \eqref{integral equation}.
As $\alpha\in C^2_c([0,\infty)$, we have $\beta\in C^1_c([0,\infty))$, and the equation takes the form 
$$-\kappa_n \int_t^\infty \beta'(r)\,\phi'(r)^{n-j}\d r=\alpha(t)$$
for $t> 0$. Consequently,
$$
\beta(t)=-\frac1{\kappa_n}\int_t^\infty\frac{\alpha'(r)}{\phi'(r)^{n-j}}\d r
$$
for $t>0$. By the conditions on $\phi$ and $\alpha$, we conclude that
$$
\lim_{t\to0^+}\beta(t)=+\infty,
$$
which is a contradiction.
\end{proof}

\subsection*{Acknowledgments} M. Ludwig was supported, in part, by the Austrian Science Fund (FWF):  P 34446 and 
F. Mussnig was supported by the Austrian Science Fund (FWF): J 4490-N.

\end{document}